\documentclass[11pt,twoside,letter]{article}
\usepackage{amsmath,amsthm,  amsfonts, graphicx, epsfig}
\usepackage{amssymb}

\usepackage{bbm}
\usepackage{mathrsfs}

\usepackage{mathtools}
\mathtoolsset{showonlyrefs}

\usepackage{epstopdf}

\usepackage{enumerate}

\usepackage[colorlinks=true, pdfstartview=FitV, linkcolor=blue,
citecolor=blue, urlcolor=blue]{hyperref}
\usepackage[usenames]{color}

\usepackage{url}
\makeatletter\def\url@leostyle{%
	\@ifundefined{selectfont}{\def\UrlFont{\sf}}{\def\UrlFont{\scriptsize\ttfamily}}} \makeatother

\usepackage[margin=.95in, letterpaper]{geometry}

\newtheorem{theorem}{Theorem}
\newtheorem{proposition}[theorem]{Proposition}
\newtheorem{lemma}[theorem]{Lemma}

\theoremstyle{definition}
\newtheorem{definition}[theorem]{Definition}
\newtheorem{example}[theorem]{Example}
\theoremstyle{remark}
\newtheorem{remark}[theorem]{Remark}

\numberwithin{equation}{section}
\numberwithin{theorem}{section}

\def\cF{\mathcal{F}}

\def\cH{\mathcal{H}}

\def\cP{\mathcal{P}}
\def\cQ{\mathcal{Q}}

\def\cT{\mathcal{T}}
\def\cU{\mathcal{U}}

\def\cX{\mathcal{X}}
\def\cY{\mathcal{Y}}
\def\cZ{\mathcal{Z}}

\def\bR{\mathbb{R}}

\def\bTheta{{\boldsymbol \Theta}}

\newcommand{\wt}{\widetilde}
\newcommand{\wh}{\widehat}

\newcommand{\1}{\mathbbm{1}}            
\newcommand{\set}[1]{\{#1\}}            
\renewcommand{\mid}{\;|\;}              

\DeclareMathOperator*{\argmin}{arg\,min} 
\DeclareMathOperator{\supp}{\mathrm{supp}}

\newenvironment{tightlist}[1]{%
    \list{{\textup{(\roman{enumi})}}}{\settowidth\labelwidth{{\textup{(#1)}}}
    \leftmargin -6pt \advance\leftmargin\labelsep \itemindent \parindent
    \parsep 0pt plus 1pt minus 1pt \topsep 0pt \itemsep 0pt
    \usecounter{enumi}}}{\endlist}

\title{Risk Filtering and Risk-Averse Control of Markovian Systems Subject to Model Uncertainty}

\author{Tomasz R. Bielecki\,\thanks{Department of Applied Mathematics, Illinois Institute of Technology
		\newline \hspace*{1.45em}  10 W 32nd Str, Building RE, Room 220, Chicago, IL 60616, USA
		\newline \hspace*{1.45em}  Emails: \url{tbielecki@iit.edu} (T. R. Bielecki), and \url{cialenco@iit.edu} (I. Cialenco)
		\newline \hspace*{1.45em}  URLs: \url{http://math.iit.edu/\~bielecki}  and \url{http://cialenco.com}
		\vspace{0.5em}}
		\and
	Igor Cialenco\,\footnotemark[1]
\and
	Andrzej Ruszczy\'{n}ski\,\thanks{Department of Management Science
and Information Systems, Rutgers University,
		\newline \hspace*{1.45em} 100 Rockafeller Road,  Room 5182, Piscataway, NJ 08854, USA
		\newline \hspace*{1.45em} Email: \url{rusz@business.rutgers.edu}, URL: \url{http://www.rusz.rutgers.edu/}
		\vspace{0.5em}}
}

\date{First circulated: June 18, 2022}

\begin{document}

\maketitle

\begin{abstract}
We consider a Markov decision process subject to model uncertainty in a Bayesian framework, where we assume that the state process is observed but its law is unknown to the observer. In addition, while the state process and the controls are observed at time $t$, the actual cost that may depend on the unknown parameter is not known at time $t$. The controller optimizes total cost by using a family of special risk measures, that we call risk filters and that are appropriately defined to take into account the model uncertainty of the controlled system. These key features lead to non-standard and non-trivial risk-averse control problems, for which we derive the Bellman principle of optimality. We illustrate the general theory on two practical examples: optimal investment and clinical trials.
\end{abstract}

\section{Introduction}

We study a risk-averse Markov decision problem (MDP) subject to uncertainty about the underlying dynamics as well as uncertainty about the risk-averse criterion.

Literature concerning risk-averse MDPs is rather abundant, and we refer to e.g. \cite{FanRuszczynski2014,FanRuszczynski2018} and references therein. Similarly, there is a vast literature on MDPs subject to model uncertainty, and we refer to \cite{BCCCJ2019} for an overview of the classical methodologies on this topic.
However, to the best of our knowledge, the present study is the first systematic study of risk-averse MDPs subject to model uncertainty.
The earlier effort in \cite{lin2021bayesian} focuses on the CVaR criterion that has an equivalent expected value formulation. It needs to be stressed that we are not only concerned with uncertainty regarding the underlying dynamics, but also uncertainty about the optimization criterion, which is a novel and important practical feature, as two examples below show. While frequent in machine learning literature, mainly concerned with the expected value criterion, such as \cite{lattimore2020bandit,sutton2018reinforcement}, it has not been addressed in the risk-averse case.

The Knightian uncertainty that we consider is parametric in nature, and our approach to tackle the respective MDP is rooted in the Bayesian methodology. It came to us as quite a surprise that accounting for possible uncertainty about the optimization criterion leads to rather intricate conceptual ideas and technical manipulations.
In order to avoid measurability and integrability issues that are notorious  and intrinsic in MDPs on general state and action spaces, and that would quite likely burden the main takeaways from this study, we decided to work with discrete state, action and parameter spaces. However, morally, the results should hold true in much more generality, that will be addressed in future works. We chose to use integral notation with respect to the state variables, which is much more pleasing to the eye and lighter then the summation notation. We keep the summation with respect to the time variable though, whenever needed.

The solution to the considered risk-averse MDP hinges on the key and new concepts of dynamic risk filters and recursive dynamic risk filters. This, in particular allows to derive a version of the dynamic programming routine suited to the needs of our uncertain risk-averse MDP.

The paper is organized as follows. In Section~\ref{s:pomdp} we set the stage and define MDP and the model uncertainty framework. Also here we introduce a series of probability measures and some of their properties used frequently in the sequel. Section~\ref{s:filters} is devoted to risk filters, starting with the definition and some fundamental properties of these objects. The key concept of parameter consistency of risk filters is introduced in Section~\ref{s:param-consist}, while the time consistency of risk filters is studied in Section~\ref{s:time-consist}. In this section we provide a characterization of parameter consistent and time consistent risk filters; cf. Theorem~\ref{t:filter-structure}.
Also here we discuss two important examples of risk filters: expectation of an additive functional, Example~\ref{ex:additive-reward}, and risk-sensitive criteria in the context of clinical trials, Example~\ref{ex:3.6}. The structure Theorem~\ref{t:filter-structure} leads to the notation of recursive risk filters, introduced in Section~\ref{s:bayes-op}. Section~\ref{sec:med} is devoted to the risk-averse control problem. In Section~\ref{s:Bayes-Operator2} we derive the Bayes operator for the posterior of the parameter of interest. Then, we derive the dynamic programming backward recursion for the classical additive reward case; Section~\ref{sec:add}. Here, as a particular case, we briefly discuss the optimal investment and consumption problem, when the investor faces Knightian uncertainty and unknown risk-aversion parameter; Example~\ref{ex:risk-averse-unknown}. We conclude with the solution to the optimal control problem for a general recursive risk filer: Theorem~\ref{thm:RA}.

Finally, we want to mention that while writing this manuscript we strove to keep a balance between heavy notations and rigor. Nevertheless, some formulas still may appear overwhelming, which is typically the case for MDPs.

\section{Markov Decision Processes with Model Uncertainty}\label{s:pomdp}

We consider an observed, controlled random process $X=\{X_t\}_{t=1 ,\dots, T}$.
The corresponding state space is a finite set $\cX$.
The underlying probability  space that we will work with is canonical. It includes the space of paths of $X$:
$\varOmega=\underbrace{\cX\times \cdots \times \cX}_{T\text{ times}}=(\cX)^T$,  endowed with the canonical product $\sigma$-field
$\cF=\underbrace{2^\cX \otimes \cdots \otimes 2^\cX}_{T\text{ times}}$. The elements of $\Omega$ are $\omega=(\omega_1,\ldots,\omega_T)$.
We use $x_t$ to denote the canonical projections at time $t$, so that $X_t (\omega) = x_t =\omega_t$.
We let $\{\cF_t^X\}_{t=1,\dots,T}$ to denote the canonical filtration generated by the process $X$,  so that $\cF_t^X=\underbrace{2^\cX \otimes \cdots \otimes2^\cX}_{t\text{ times}}\otimes \underbrace{\{\Omega,\emptyset\} \otimes \cdots \otimes \{\Omega,\emptyset\}}_{T-t\text{ times}}$. We will make use of the notations $\cT=\set{1,\ldots,T}$ and $\cT_t=\set{t,\ldots,T}$.

The control space is  given by a finite set $\cU$, and  the set of admissible controls at step $t$ is given by  a multifunction $\cU_t : \cX \rightrightarrows \cU$ with nonempty values. We consider a parametric family of transition kernels
$K_\theta: \cX \times \cU \rightarrow \cP(\cX)$,
where $\cP(\cX)$ is the space of probability measures on $\cX$, and $\theta \in \boldsymbol \Theta$ represents an unknown parameter. Here, $\boldsymbol \Theta$ is a finite set.
The unknown true value of the parameter $\theta$ is $\theta^*$.

We will consider  the Bayesian setting, and therefore, we consider the product space $\widehat{\varOmega} = \varOmega \times \boldsymbol \Theta$ endowed with product $\sigma$-algebra $\widehat \cF:=\cF\otimes 2^{\boldsymbol \Theta}$. We denote $\widehat{\omega}=(\omega,\theta)$ and $\widehat{\omega}_t=(\omega_t,\theta)$. In accordance with the Bayesian setting we denote by $\Theta$ a random variable on $(\widehat{\varOmega},\widehat \cF)$ with values in $\boldsymbol \Theta$, and  with $\Theta(\widehat \omega)=\theta$. We also assume that some \emph{prior distribution} $\xi_1$ of $\Theta$ (supported in $\boldsymbol \Theta$) is available.

The process $\{X_t\}_{t\in\cT}$ considered as a process on $(\widehat{\varOmega},\widehat \cF)$ is denoted as $\widehat X=\{\widehat X_t\}_{t\in\cT}$, and $\widehat X_t(\widehat \omega)=X_t(\omega)$. Accordingly, the canonical filtration  generated by the process $\widehat X$ is given as $\{\cF_t^{\widehat X}=\cF_t^X\otimes 2^{\boldsymbol \Theta},\ t\in\cT \}$.

At time~$t$,
the history of observed states is
$h_t = (x_1, x_2, \dots, x_t)$,
while all the information available for making a decision is
$g_t = (x_1, u_1, x_2, u_2, \dots, x_t)$.
We use
$
\cH_t := \cX^t = \underbrace{\cX\times \cdots \times \cX}_{t\text{ times}}
$
to denote the spaces of possible state histories $h_t$.
We make distinction of $g_t$ and $h_t$ because
we should make decision of $u_t$ based on $g_t$ as the past controls $u_1, \dots, u_{t-1}$ are also taken into consideration when estimating the conditional distribution of $\theta$. We write $H_t$ for $(X_1,\dots,X_t)$ and $\widehat H_t$ for $(\widehat X_1,\dots,\widehat X_t)$.

A \emph{history-dependent admissible policy} $\pi=(\pi_1,\dots,\pi_{T})$ is a sequence of  functions $\pi_t(g_t)$ such that $\pi_t(g_t) \in \cU_t(x_t)$ for all possible $g_t$. One can easily prove that for such an admissible policy $\pi$, each $\pi_t$ reduces to a  function of $h_t=(x_1, x_2,\dots, x_t)$,\footnote{We are still using $\pi_s$ to denote the decision rule; it will not lead to any misunderstanding.}
as $u_s = \pi_s(x_1,\dots,x_s)$ for all $s=1 ,\dots, t-1$.
Therefore the set  of admissible policies is
\[
\varPi =  \ \big\{ \, \pi = (\pi_1,\dots,\pi_{T}) \, :  \, \; \pi_t(x_1,\dots,x_t) \in \cU_t(x_t), \ t\in\cT \,\big\}.
\]
Any policy $\pi \in \varPi$ defines the control process, also denoted by $\pi=\{\pi_t\}_{t\in\cT}$, with $\pi_t=\pi_t(X_1,\ldots,X_t)$. We make a distinction between $u_t=\pi_t(x_1,\ldots, x_t)$ and $\pi_t=\pi_t(X_1,\ldots,X_t)$.

As said in the Introduction, even though we work with discrete spaces $\cX$ and $\boldsymbol \Theta$, we are using the more convenient integral notation, rather than the summation notation.


For a fixed initial state $x_1$, every policy $\pi\in \varPi$, and every $\theta\in \boldsymbol \Theta$, a probability measure $P^{\pi}_{\theta}$ on $(\varOmega,\cF)$ is uniquely defined 
by:
\begin{align}
P^{\pi}_{\theta}(A_1\times A_2\times \ldots &\times A_{T-1}\times A_T)
 =\int_{A_1}\int_{A_2}\cdots \int_{A_{T-1}}K_\theta(A_T|x_{T-1},\pi_{T-1}(x_1,\ldots, x_{T-1}))\nonumber \\
 &\times K_\theta(dx_{T-1}|x_{T-2},\pi_{T-2}(x_1,\ldots, x_{T-2}))\times \cdots \nonumber \\
 &\quad \cdots \times  K_\theta(dx_2|x_1,\pi_1(x_1))\delta_{x_1}(dy), \qquad A_t\subset \cX,\ t\in\cT, \label{eq:I-T-1}
\end{align}
where, as usual, $\delta_x$ denotes the Dirac measure concentrated at $x$.
In particular,
\[
P^{\pi}_{\theta}(A)=P^{\pi}_{\theta}(X\in A)=P^{\pi}_{\theta}(\{\omega \in \Omega\, :\, X(\omega)\in A\}),\quad A\subset \Omega.
\]

The true but unknown measure under the policy $\pi$ is $P^{\pi}_{\theta^*}$. This measure gives the true law of the canonical process $X$ subject to control strategy $\pi$.

Given the prior distribution $\xi_1$, a probability measure $P^{\pi}$ on $(\widehat{\varOmega},{\widehat \cF})$ is defined as well:
\begin{align}\label{eq:I-T-integrated}
P^{\pi}(A\times D)=\int_D\, P^{\pi}_{\theta}(A)\;\xi_1(d\theta),\quad A\subset \Omega,\ D\subset \boldsymbol \Theta.
\end{align}
In particular,
\[
P^{\pi}(A\times D)=P^{\pi}(\{\widehat \omega \in \widehat \Omega\, :\, \widehat X(\widehat \omega)\in A,\, \Theta(\widehat \omega)\in D  \}).
\]
Clearly, $\xi_1$ is the marginal of $P^{\pi}$, that is $\xi_1(D)=P^{\pi}(\Omega\times D)$. To simplify the ensuing study, we assume that for any $t\in\cT$ and $h_t\in\cH_t$ we have $P^{\pi}(\widehat H_t=h_t)>0.$ This assumption is of course an assumption about the kernels $K_\theta$, $\theta\in \boldsymbol \Theta$.

Furthermore, for each $t=1,\dots,T-1$ and for each history $h_t \in \cH_t$, we define the set of \textit{tail control strategies}
\[
\Pi^{t,h_t}=\{\pi^{t,h_t} : \pi_t^{t,h_t}=\pi_t(h_t),\ \pi_s^{t,h_t}(x_{t+1},\dots,x_s)=\pi_s(h_t,x_{t+1},\dots,x_s),\ s\in\cT_{t+1},\ \pi \in \Pi\}.
\]
In addition, for $t=1,\ldots, T-1$,  and for each $\theta\in \boldsymbol \Theta$, $h_t\in \cH_t$, and $\pi^{t,h_t}\in \Pi^{t,h_t}$ we construct a probability  measure $P^{\pi^{t,h_t}}_{\theta,t+1,T}$ on $\cX^{T-t}$ in analogy to \eqref{eq:I-T-1}.
Specifically, we put
\begin{align}
P^{\pi^{t,h_t}}_{\theta,t+1,T}& (A_{t+1}\times \cdots \times A_T)
	=\int_{A_{t+1}}\int_{A_{t+2}}\cdots \int_{A_{T-1}}K_\theta(A_T|x_{T-1},\pi_{T-1}(h_t, x_{t+1},\ldots, x_{T-1})) \nonumber \\
	&\cdot  K_\theta(dx_{T-1}|x_{T-2},\pi_{T-2}(h_t,x_{t+1},\ldots, x_{T-2})) \cdots K_\theta(dx_{t+2}|x_{t+1},\pi_{t+1}(h_{t},x_{t+1}))  \nonumber \\
	& \qquad \cdot K_\theta(dx_{t+1}|x_{t},\pi_t(h_{t})), \qquad A_s\subset \cX,\ s\in\cT_{t+1}. \label{eq:I-T}
\end{align}

We proceed with three technical results that are rather straightforward consequences of the above set-up.

\begin{lemma}\label{lemma:21}
	 For any $h_t\in\cH_t$, and $A_s\subset \cX, \ s\in\cT_{t+1}$, $\pi \in \Pi$, and the corresponding $\pi^{t,h_t}\in\Pi^{t,h_t}$ we have that
\begin{equation}\label{eq:L21}
P^{\pi^{t,h_t}}_{\theta,t+1,T}(A_{t+1}\times \ldots \times A_T)=
P^\pi_\theta(X_{t+1}\in A_{t+1},\ldots,X_{T}\in A_{T}|H_t=h_t).
\end{equation}
\end{lemma}
\begin{proof}
First, note that\footnote{To further simplify the notation we write $K_\theta(x|...)$ in place of $K_\theta(\{x\}|...)$. In a similar way, for a probability measure $Q$ on $\cQ$, and $y\in\cQ$, we may write $Q(y)$ instead of $Q(\set{y})$.}
$$
P_\theta^\pi(X_1=x_1,\ldots,X_t=x_t) = K_\theta(x_2|x_1,\pi_1(x_1)) \, K_\theta(x_3|x_2,\pi_2(x_1,x_2))\cdots K_\theta(x_t|x_{t-1},\pi_{t-1}(x_1,\ldots x_t)).
$$
On  the other hand,
\begin{align*}
P^\pi_\theta & (X_{t+1}\in A_{t+1},\ldots,X_{T}\in A_{T}, H_t=h_t) = P^\pi_\theta(X_1=x_1,\ldots, X_t=x_t, X_{t+1}\in A_{t+1},\ldots,X_{T}\in A_{T}) \\
& = \int_{\set{x_1}}\ldots \int_{\set{x_t}} P_{\theta, t+1,T}^{\pi^{t, \bar h_t}}(A_{t+1}\times\cdots\times A_T)  K_\theta(d\bar x_t| \bar x_{t-1}, \pi_{t-1}(\bar h_{t-1})) \cdots K_\theta(d \bar x_2|\bar x_1,\pi_1(\bar x_1)) \delta_{x_1}(y)  \\
& = P_{\theta, t+1,T}^{\pi^{t, h_t}}(A_{t+1}\times\cdots\times A_T)  K_\theta(x_t| x_{t-1}, \pi_{t-1}( h_{t-1})) \cdots K_\theta(x_2| x_1,\pi_1( x_1)).
\end{align*}
Combining the above we immediately have \eqref{eq:L21}.
\end{proof}

\medskip

For future reference we denote by $P^{\pi^{t,h_t}}_{\theta,t+1}$ a measure on $(\mathcal{X},2^\mathcal{X})$ defined as
\begin{equation}\label{eq:new-measure}
	P^{\pi^{t,h_t}}_{\theta,t+1}(B)=P^{\pi^{t,h_t}}_{\theta,t+1,T}(B\times \mathcal{X}^{T-t-1}) = P^{\pi}_{\theta, t+1,T}(X_{t+1}\in B).
\end{equation}
Thus, we have that, for $t\leq T-1$,
\begin{equation}\label{eq:P-theta-t+1}
P_{\theta, t+1}^{\pi^{t,h_t}}(B) = \int_B\, K_\theta(dx_{t+1} \mid x_t, \pi_t(h_t))=K_\theta(B \mid x_t, \pi_t(h_t))=
P^\pi_\theta(X_{t+1}\in B|H_t=h_t).
\end{equation}
Next, we construct a probability measure $P^{\pi^{t,h_t}}_{t+1,T}$ on $\cX^{T-t}\times \boldsymbol \Theta$ as
\begin{align}\label{eq:I-T-integrated2}
P^{\pi^{t,h_t}}_{t+1,T}(A\times D)=\int_D\, P^{\pi^{t,h_t}}_{\theta,t+1,T}(A)\;\xi^{\pi,h_t}_{t}(d\theta),\quad A\in \underbrace{2^{\cX} \otimes \cdots \otimes 2^{\cX}}_{T-t\text{ times}},\ D\in 2^{\boldsymbol \Theta},
\end{align}
where $\xi^{\pi,h_t}_{t}\in \cP(\boldsymbol \Theta)$, is given as
\begin{equation}\label{eq:xi-tCond}
\xi^{\pi,h_t}_{t}(D)= P^\pi(\Theta \in D \, | \, \widehat H_t = h_t),\ \textrm{for} \ t=2,\ldots,T, \quad \textrm{and} \quad \xi^{\pi,h_1}_{1}(D)=\xi_1(D).
\end{equation}
We note that we have the following indenty for the conditional measure
\begin{equation}\label{ii-1}
P_{t+1,T}^{\pi^{t,h_t}} (A\mid \Theta=\theta) = P_{\theta, t+1,T}^{\pi^{t,h_t}}(A), \quad \theta\in\boldsymbol{\Theta}, \ A\in 2^{\cX} \otimes \cdots \otimes 2^{\cX}.
\end{equation}

\begin{lemma} \label{lemma:2.2}
Let $A\in \underbrace{2^{\cX} \otimes \cdots \otimes 2^{\cX}}_{T-t\text{ times}}$, $D\in 2^{\boldsymbol \Theta}$, $h_t\in\cH_t$, and $\pi\in\Pi$. Then:\\
(i) We have
\begin{equation}\label{i-1}
P^{\pi^{t,h_t}}_{t+1,T}(A\times D)=
P^\pi((\widehat X_{t+1},\widehat X_{t+2},\ldots,\widehat X_{T-1},\widehat X_{T})\in A,\boldsymbol \Theta\in D|\widehat H_t=h_t).
\end{equation}
\end{lemma}
\begin{proof}
First, in view of \eqref{eq:xi-tCond} and \eqref{eq:I-T-integrated} we note that
\begin{equation*}
\xi_t^{\pi,h_t} (D ) = \frac{\int_D P_\theta^\pi(H_t=h_t)\; \xi_1(d\theta)}{P^\pi(\widehat H_t=h_t)}.
\end{equation*}
Thus,
\begin{equation}\label{eq:L22-2}
P^{\pi^{t,h_t}}_{t+1,T}(A\times D)=  \frac{\int_D  P^{\pi^{t,h_t}}_{\theta,t+1,T}(A) P_\theta^\pi(H_t=h_t) \;\xi_1(d\theta)}{ P^\pi(\widehat H_t=h_t)}.
\end{equation}
On the other hand, using \eqref{eq:I-T-integrated} and Lemma~\ref{lemma:21}, we have
\begin{align*}
& P^\pi((\widehat X_{t+1},\widehat X_{t+2},\ldots,\widehat X_{T-1},\widehat X_{T})\in A, \Theta\in D|\widehat H_t=h_t)\\
& = \frac{\int_D P_\theta^\pi((X_{t+1}, \ldots, X_T)\in A, H_t=h_t) \;\xi_1(d\theta)}{P^\pi (\widehat H_t=h_t)} \\
& =  \frac{\int_D P_\theta^\pi((X_{t+1}, \ldots, X_T)\in A \mid H_t=h_t) P^\pi_\theta(H_t=h_t)\; \xi_1(d\theta)}{P^\pi (\widehat H_t=h_t)} \\
& = \frac{\int_D P_{\theta, t+1}^{\pi^{t,h_t}}(A) P^\pi_\theta(H_t=h_t) \;\xi_1(d\theta)}{P^\pi (\widehat H_t=h_t)}.
\end{align*}
This, combined with \eqref{eq:L22-2} concludes the proof of part (i).

\end{proof}

\begin{remark}
Formally, taking $T=t+1$ in \eqref{ii-1} we obtain
\begin{equation}\label{ii-1-new}
P^{\pi^{t,h_t}}_{t+1,t+1|\theta}(A)=P^{\pi^{t,h_t}}_{\theta,t+1,t+1}(A)=P^{\pi^{t,h_t}}_{\theta,t+1}(A),
\end{equation}
where in the last equality we used \eqref{eq:new-measure}.

\end{remark}

\begin{lemma}\label{cacy-cacy} Let $t\in \set{1,\ldots,T-1}$, and let
 $F$ be a function on $\cX^{T-t} \times \bTheta$. Then, for each $h_t\in \cH_t$ we have
\begin{align}
E^\pi [ &F(\widehat X_{t+1},\widehat X_{t+2},\ldots,\widehat X_{T-1},\widehat X_{T},\Theta)  \mid \widehat H_t=h_t] = \nonumber \\
& \int_\bTheta \int_{\cX^{T-t}} F(x_{t+1},\ldots,x_T,\theta)P^{\pi^{T-1,h_{T-1}}}_{\theta,T}(dx_{T})\cdots P^{\pi^{t,h_t}}_{\theta,t+1}(dx_{t+1})\xi^{\pi,h_t}_{t}(d\theta),
\label{eq:condExpect}
\end{align}
where $E^\pi$ denotes the expectation with respect to probability $P^\pi$.
\end{lemma}
\begin{proof}
In view of Lemma~\ref{lemma:2.2}, we have
$$
P^\pi(dx_{t+1}, \ldots, dx_T; d\theta \mid \widehat{H}_t=h_t) = P^{\pi^{t,h_t}}_{t+1,T} (dx_{t+1},\ldots,dx_T; d\theta).
$$
Consequently, by \eqref{eq:I-T-integrated2}, we continue
$$
P^{\pi^{t,h_t}}_{t+1,T} (dx_{t+1},\ldots,dx_T; d\theta) = P^{\pi^{t,h_t}}_{\theta,t+1,T}(dx_{t+1},\ldots,dx_T) \xi_t^{\pi^{t,h_t}} (d\theta).
$$
This combined with  \eqref{eq:I-T} and \eqref{eq:P-theta-t+1} yields the identity \eqref{eq:condExpect}.
\end{proof}

For future reference we denote by $P^{\pi^{t,h_t}}_{t+1}$ the measure on $\cX \times \boldsymbol \Theta$ defined as
\begin{equation}\label{eq:new-measure-1}
P^{\pi^{t,h_t}}_{t+1}(B\times D)=P^{\pi^{t,h_t}}_{t+1,T}(B\times \mathcal{X}^{T-t-1}\times D)= \, P^\pi(\widehat X_{t+1}\in B, \Theta \in D \, | \, \widehat H_t =  h_t),
\end{equation}
for $t=1,\ldots,T-1$, with the convention, employed throughout, that $B\times \mathcal{X}^{0}\times D=B\times D.$

For $t=T$ and $h_T\in \cH_T$ we construct a measure $P^{\pi^{T,h_T}}_{T+1,T}$ on $\boldsymbol \Theta$ as
\[
P^{\pi^{T,h_T}}_{T+1,T}(D) = P^\pi(\Theta \in D \, | \, \widehat H_T = h_T)=\xi^{\pi,h_T}_{T}(D).
\]

Given that a strategy $\pi$ is used, then at each time $t\in\cT$ a random cost $Z^{\pi}_{\theta^*,t}$
is incurred, with
\[
 Z^{\pi}_{\theta,t} =c_t(X_t,\pi_t,\theta),
\]
where $c_t:\cX\times\cU\times\boldsymbol \Theta\to \bR_+$.

\begin{remark}
It is important to note that even though $X_t$ and $\pi_t$ are observed at time $t$,  the actual cost $c_t(X_t,\pi_t,\theta^*)$ is not known {(or observed) at time $t$} as $\theta^*$ is not known. The dependence of both the transition kernel and the accrued costs on the unknown parameter is an important practical situation, leading to non-standard and non-trivial risk-averse Markov decision problems.
\end{remark}

To proceed, for each $t\in\cT$  and each history $h_t \in \cH_t$ we denote
\begin{equation}\label{eq:cost1}
Z^{\pi,h_t}_{\theta,t,t} =c_t(x_t,\pi_t(h_t),\theta),
\end{equation}
and for each $s=t+1,\dots,T$  we put
\begin{equation}\label{eq:cost2}
 Z^{\pi,h_t,x_{t+1},\ldots,x_s}_{\theta,t,s} =c_s(x_s,\pi^{t,h_t}_s(x_{t+1},\ldots,x_s),\theta).
\end{equation}

Note that, for a fixed strategy $\pi$ and a fixed  $h_t \in \cH_t$, we have that $c_t(x_t,\pi_t(h_t),\cdot)$ is a function on $\boldsymbol \Theta$, and
$c_s(\cdot,\pi^{t,h_t}_s(\cdot,\ldots,\cdot),\theta)$ is a function on $\cX^{s-t} \times \boldsymbol \Theta$.

\section{Risk Filters for MDPs with Model Uncertainty}
\label{s:filters}

\subsection{Dynamic risk filters}

For $t=1,\ldots,T-1$ and $s=t,\ldots,T$,  we denote by $\cZ_t^\mathcal{X}$ and  $\cZ_{t,s}$  the spaces of real valued functions on $\cX^t$ and $\cX^{s-t} \times \boldsymbol \Theta$, respectively, where $\cX^{0} \times \boldsymbol \Theta:=\boldsymbol \Theta$, so that $\cZ_{t,t}$ is the space of real valued functions on $\boldsymbol \Theta$. For  $Z_{t,s},W_{t,s}\in \cZ_{t,s}$, the comparison between these functions is understood point-wise;  $Z_{t,s} \le W_{t,s}$ means that  $Z_{t,s}(x_{t+1},\ldots,x_s,\theta) \le W_{t,s}(x_{t+1},\ldots,x_s,\theta) $ for all $(x_{t+1},\ldots,x_s,\theta) \in \cX^{s-t} \times \boldsymbol \Theta.$

For any policy $\pi \in \varPi$, our objective is to evaluate at each time $t\in\cT$, the riskiness of the sequence of costs $Z^{\pi,h_t}_{\Theta,t},Z^{\pi,h_t,X_{t+1}}_{\Theta,t,t+1},  \dots, Z^{\pi,h_t,X_{t+1},\ldots,X_T}_{\Theta,t,T}$, given history $h_t$, in such a way that the evaluation is $\cF_t^X$-measurable.
We denote  by
\[
{\cZ}^{t,T} = {\cZ}_{t,t}\times {\cZ}_{t,t+1}\times \dots \times {\cZ}_{t,T}
\]
the space of \textit{conditional cost functions}\footnote{The term \textit{conditional} refers to the fact that at any time $t$ we consider cost functions that depend on a history $h_t$.} in periods $t,\dots,T$.

For $t=1,\ldots,T$ and $s=t,\ldots,T$ we also use $\cP_{t,s}$ to denote the space of probability measures on the space of paths starting at time $t$
and ending at time $s$, and on  realizations of the parameter $\theta$, that is on the space $\cX^{s-t+1}\times \boldsymbol \Theta$. Additionally,   we understand $\cP_{T+1,T}$ as $\cP(\boldsymbol \Theta)$, because no future paths are possible. Note, in particular, that $P^{\pi^{t,h_t}}_{t+1,T}\in \cP_{t+1,T}$, for $t\in\cT$, and
$P^{\pi^{t,h_t}}_{t+1,t+1}\in \cP_{t+1,t+1}$.

Observe that at time $t=1,\ldots,T-1$ we know the history $h_t$, and, for any policy $\pi \in \varPi$  (in principle), we can evaluate the distribution
of $(X_{t+1},\dots,X_T,\Theta)$ under the measure  $P^{\pi^{t,h_t}}_{t+1,T}$.

We proceed with stating three key definitions.

\begin{definition} \label{def:CRE}
For a fixed $t\in\cT$, a mapping $\rho_t: {\cZ}^{t,T} \times \cP_{t+1,T}\to  \bR$, is called a \emph{conditional risk filter}.
\end{definition}	
Note that, in particular, for any $(Z_{t,t},\dots,Z_{t,T})\in {\cZ}^{t,T}$ and $\pi \in \Pi$ we have
	\[
	\rho_t(Z_{t,t},\dots,Z_{t,T};  P_{t+1,T})=R(h_t), \ \textrm{for all}\ h_t\in \mathcal{X}^t,
	\]
	for some function $R: \mathcal{X}^t \to \bR$.

\begin{definition} \label{basic-prop-pomdp} Let $t\in\cT$. A conditional risk filter $\rho_t$  
\begin{tightlist}{iv}
\item  is \emph{normalized} if $\rho_t(0,0,\dots,0; P_{t+1,T})=0$ for all $P_{t+1,T}\in \cP_{t+1,T}$;

\item is \emph{monotonic} if $\rho_t(Z_{t,t},\dots,Z_{t,T};  P_{t+1,T}) \le \rho_t(W_{t,t},\dots,W_{t,T};  P_{t+1,T})$
for all  $P_{t+1,T}\in \cP_{t+1,T}$, and all $(Z_{t,t},\dots,Z_{t,T})$ and $(W_{t,t},\dots, W_{t,T})$
in  ${\cZ}^{t,T}$, such that $Z_{t,s} \le W_{t,s}$ for all $s\in\cT_{t}$;

\item  is \emph{translation invariant} if  for all $(Z_{t,t},\dots,Z_{t,T}) \in {\cZ}^{t,T}$, all $V\in  \bR$, and all $P_{t+1,T}\in \cP_{t+1,T}$,
		$$
		\rho_t(V+ Z_{t,t},Z_{t,t+1},\dots,Z_{t,T};  P_{t+1,T})=V+ \rho_t(Z_{t,t},Z_{t,t+1},\dots,Z_{t,T}; P_{t+1,T});
		$$
\item has the \emph{support property}, if
$$
\rho_t\big(Z_{t,t},\dots,Z_{t,T};  P_{t+1,T}\big) =
\rho_t\big(Z_{t,t}\1_{\text{\rm supp}_{t}(P_{t+1,T})},\dots,Z_{t,T}\1_{\text{\rm supp}_{T}(P_{t+1,T})};  P_{t+1,T}\big),
$$
for all $(Z_{t,t},\dots,Z_{t,T}) \in {\cZ}^{t,T}$, and all $P_{t+1,T}\in \cP_{t+1,T}$, and where $\supp_{s}(P_{t+1,T})$ denotes the projection of $\supp(P_{t+1,T})$ on $\cX^{s-t} \times \boldsymbol \Theta$, for $s\geq t$.
\end{tightlist}
\end{definition}

\begin{remark}\label{rem:locality}
	Let $s=t,\ldots,T$ and let $\set{Z^y_{s,T}, y\in\cY}$ be a family of functions parameterized by $y$, for some non-empty set $\cY$. Then, by the normalization property, for any $A\subset\cY$,  $y\in\cY$,  and $P\in\cP_{t+1,T}$, we have that
$$
\1_A(y) \rho_t(Z^y_{t,T}, \ldots, Z^y_{T,T}; P) =  \rho_t(\1_A(y)Z^y_{t,T}, \ldots, \1_A(y)Z^y_{T,T}; P).
$$
\end{remark}

\begin{definition} \label{def_DRE}
	A \emph{dynamic risk filter} $\rho=\big\{\rho_t\big\}_{t \in\cT} $ is a sequence of conditional risk filters $\rho_t: {\cZ}^{t,T} \times \cP_{t+1,T}\to \bR$. We say that it is normalized, monotonic,  translation invariant, or has the support property,  if all $\rho_t$, $t\in\cT$, satisfy the respective conditions of Definition \ref{basic-prop-pomdp}.
\end{definition}

\subsection{Parameter  Consistency}\label{s:param-consist}

Let $t=1,\ldots,T-1$ and $s=t,\ldots,T$. For any probability measure $P_{t,s}\in\cP_{t,s}$, we denote by $P_{t,s|\Theta}(\cdot,\cdot)$, the stochastic kernel from $\boldsymbol \Theta$ to $\mathcal{X}^{s-t+1}$ defined as
\begin{equation}\label{eq:ker}
P_{t,s|\Theta}(\theta,A)=\frac{P_{t,s}(A\times \{\theta\})}{P_{t,s}(\mathcal{X}^{s-t+1}\times \{\theta\})},\quad A\subset \mathcal{X}^{s-t+1},\ \theta \in \boldsymbol \Theta.
\end{equation}
The corresponding marginal on $\boldsymbol{\Theta}$ is denoted by $P_{t,s,\Theta}$, so that
\begin{equation}\label{eq:marg}
P_{t,s, \Theta}(D)=P_{t,s}(\mathcal{X}^{s-t+1}\times D),\quad D \subset \boldsymbol \Theta.
\end{equation}
Clearly, the measure $P_{t,s}$ admits disintegration
\[
P_{t,s}(A\times B)=\int_B\, P_{t,s|\theta}(A) P_{t,s,\Theta}(d\theta)=:P_{t,s,\Theta } \circledast P_{t,s|\Theta}(A\times B),
\]
where we use a simplified notation
\begin{equation}\label{eq:simple}
P_{t,s|\theta}(A):=P_{t,s|\Theta}(\theta,A).
\end{equation}
We note that for any stochastic kernel $\kappa_{s,t}(\cdot,\cdot)$ from $\boldsymbol \Theta$ to $\mathcal{X}^{s-t}$ and for any probability measure $\mu$ on $2^{\boldsymbol \Theta}$ one can construct a unique probability measure on the product space $\cX^{s-t+1}\times \boldsymbol{\Theta}$ as
\[
m_{t,s}(A\times B)=\int_B\, \kappa_{t,s}(\theta,A)\; \mu(d\theta)=:\mu \circledast \kappa_{t,s}(A\times B).
\]
In particular, with $\mu=\delta_\theta$ and $\kappa_{t,s}=P_{t,s|{\Theta}}$, with $P_{t,s}\in\cP_{t,s}$, we get
\begin{equation}\label{eq:delta}
m_{t,s}(A\times B)=\delta_{ \theta} \circledast P_{t,s|{\Theta}}(A\times B)=P_{t,s|\theta}(A)\1_{B}( \theta)= P_{t,s|\theta}(A)\delta_{ \theta}(B).
\end{equation}

\begin{remark}\label{rem:formal}
In our convention, $P_{T+1,T}(\cdot)$ is a measure on $\bTheta$. This means that, formally,  $P_{T+1,T,\boldsymbol \Theta}=P_{T+1,T}$ and $P_{T+1,T|\boldsymbol \Theta}\equiv 1$,  in which case (formally)
$$
\delta_{\theta} \circledast P_{T+1,T|\Theta}=\delta_{\theta} .
$$
\end{remark}

\begin{example}\label{ex:3.6}
Fix $t\in\cT$, $h_t\in \mathcal{H}_t$ and  $\pi \in \Pi$. Take $P_{t+1,T}=P^{\pi^{t,h_t}}_{t+1,T}\in  \cP_{t+1,T}$ and $P_{t+1,t+1}=P^{\pi^{t,h_t}}_{t+1,t+1}\in  \cP_{t+1,t+1}$. Then,
\begin{equation}\label{three-eqs}
P_{t+1,T|\theta}=P^{\pi^{t,h_t}}_{\theta,t+1,T},\quad P_{t+1,t+1|\theta} = P_{\theta, t+1}^{\pi^{t,h_t}},\quad P_{t+1,T,\Theta} = \xi_t^{\pi^{t, h_t}}.
\end{equation}
The first equality above comes from \eqref{ii-1}. The second one comes from \eqref{ii-1-new}. The third one is just \eqref{eq:I-T-integrated2} with $A=\mathcal{X}^{T-t} $. Note that \eqref{eq:delta} and \eqref{three-eqs} imply that
\begin{equation}\label{eq:delta-1}
\delta_{ \theta} \circledast P^{\pi^{t,h_t}}_{t+1,T|{\Theta}}(A\times B)= \delta_{ \theta} \otimes P^{\pi^{t,h_t}}_{\theta,t+1,T}(A\times B),\quad
\delta_{ \theta} \circledast P^{\pi^{t,h_t}}_{t+1,t+1|{\Theta}}(A\times B)= \delta_{ \theta} \otimes P^{\pi^{t,h_t}}_{\theta,t+1,t+1}(A\times B).
\end{equation}

\end{example}

We introduce the following key concept.

\begin{definition} \label{cond-con}

A conditional risk filter $\rho_t: {\cZ}^{t,T}\times \cP_{t+1,T} \to \bR$ is \emph{parameter consistent}, if
for all $(Z_{t,t},\break \dots,Z_{t,T})$, $(W_{t,t},\dots,W_{t,T})\in\cZ^{t,T}$, and all $P_{t+1,T}, Q_{t+1,T}\in \cP_{t+1,T}$
 the relations
\[
 P_{t+1,T,\Theta}= Q_{t+1,T,\Theta} \notag \\
\]
and
\begin{equation}
\rho_t\big(Z_{t,t},\dots,Z_{t,T};  \delta_\theta \circledast P_{t+1,T|\Theta}\big)
\leq \rho_t\big(W_{t,t},\dots,W_{t,T};  \delta_\theta \circledast Q_{t+1,T|\Theta}\big),\quad \textrm{for all} \ \theta \in \boldsymbol \Theta,\label{eq:paramCon1}
\end{equation}
imply that
\begin{equation}\label{eq:paramCon2}
\rho_t\big(Z_{t,t},\dots,Z_{t,T};   P_{t+1,T}\big) \leq \rho_t\big(W_{t,t},\dots,W_{t,T};   Q_{t+1,T}\big).
\end{equation}

\end{definition}
In words,  if the marginal distributions of $P$ and $Q$ on $\Theta$ are the same, and the conditional risk of $Z_{t:T}:=(Z_{t,t},\dots,Z_{t,T})$ under $P$ is not greater  than that of
$W_{t:T}:=(W_{t,t},\dots,W_{t,T})$ under $Q$ for every value of $\theta$, then the risk of $Z_{t:T}$ under $P$ should be not greater
than that of $W_{t:T}$ under $Q$.

\begin{remark}\label{rem:paramAtT}
Note that parameter consistency at $t=T$ follows from the support property, translation invariance, monotonicity, and normalization of $\rho_T$.
Indeed, first observe  that according to   Remark \ref{rem:formal} the equality $P_{T+1,T,\Theta}= Q_{T+1,T,\Theta}=1$ implies that  $P_{T+1,T}=Q_{T+1,T}$. Thus,
for any $\theta\in\boldsymbol{\Theta}$
\begin{align*}
\rho_T(Z_{T,T}, \delta_\theta) \leq  \rho_T(W_{T,T}, \delta_\theta) \quad &\Leftrightarrow \quad
\rho_T(Z_{T,T}(\theta), \delta_\theta) \leq  \rho_T(W_{T,T}(\theta), \delta_\theta) \quad \Leftrightarrow \\
Z_{T,T}(\theta) + \rho_T(0,\delta_\theta) \leq  W_{T,T}(\theta) + \rho_T(0, \delta_\theta) \quad &\Leftrightarrow \quad
Z_{T,T}(\theta) \leq W_{T,T}(\theta).
\end{align*}
By monotonicity, we have that
$$
\rho_T(Z_{T,T}, P_{T+1,T}) \leq  \rho_T(W_{T,T}, P_{T+1,T}) = \rho_T(W_{T,T}, Q_{T+1,T}).
$$
This remark is used in Proposition~\ref{prop:disintegration} and also in Theorem \ref{t:filter-structure}.
\end{remark}

We have the following risk decomposition formula.
\begin{theorem}\label{t:disintegration}
Take $t=1,\ldots,T$. If a conditional risk filter	$\rho_t: {{\cZ}^{t,T}}\times \cP_{t+1,T} \to {\bR}$ is parameter consistent, then there exists a mapping
		$\widehat \rho_t: \cZ_{t,t} \times \cP(\boldsymbol \Theta) \to \bR$ such that for all $Z_{t:T}$ and $P_{t+1,T}$,
\begin{equation}\label{disintegration}
\rho_t\big(Z_{t,t},\dots,Z_{t,T};   P_{t+1,T}\big) =
\widehat \rho_t \Big(  \set{\rho_t\big(Z_{t,t},\dots,Z_{t,T};  \delta_\theta \circledast  P_{t+1,T| \Theta}\big ),\, \theta \in \boldsymbol \Theta}; {P_{t+1,T,\Theta}}\Big).
\end{equation}
\end{theorem}

\begin{proof}
Suppose two sequences $Z_{t:T}$ and $W_{t:T}$ in $\cZ^{t,T}$, and two
measures $P_{t+1,T}$ and $Q_{t+1,T}$ in $\cP_{t+1,T}$ are such that $P_{t+1,T,\boldsymbol \Theta} = Q_{t+1,T,\boldsymbol \Theta}$ and
\[
\rho_{t}\big(Z_{t,t},\dots,Z_{t,T};  \delta_\theta \circledast P_{t+1,T| \Theta}\big)
= \rho_{t}\big(W_{t,t},\dots,W_{t,T};  \delta_\theta \circledast Q_{t+1,T|\Theta}\big),\quad \forall\,\theta \in \boldsymbol \Theta .
\]
Then it follows from Definition \ref{cond-con} that
\[
\rho_t\big(Z_{t,t},\dots,Z_{t,T};   P_{t+1,T}\big)
=\rho_t\big(W_{t,t},\dots,W_{t,T};   Q_{t+1,T}\big).
\]
This means that formula \eqref{disintegration} is true.
\end{proof}

Thus, parameter consistency allows us to disintegrate the risk filtering task into two stages.
First, we evaluate the risk in a fully observed system, with the parameter $\theta$ fixed, and then
we integrate the results by using the operator $\widehat \rho_t$, which we call the \emph{marginal risk filter}.

\begin{proposition}\label{prop:disintegration}
Take $t=1,\ldots,T$.  If the conditional risk filter $\rho_t$ is parameter consistent, normalized,  monotonic, and has the translation invariant property, or the support property, then the mapping $\widehat \rho_t(\cdot;\cdot)$ has  the corresponding properties as well (in the sense indicated in the proof below).
\end{proposition}

\begin{proof}
Indeed, consider any measure $\varLambda \in \cP(\bTheta)$. Then for any $P_{t+1,T} \in \cP_{t+1,T}$ such that
$P_{t+1,T | \Theta}=\varLambda$, we will use the formula \eqref{disintegration} to analyze the implied properties of $\widehat \rho_t$.

\noindent
1) Suppose $\rho_t$ is normalized.  Then (the symbol $\mathbf{0}$ below denotes a function on $\boldsymbol \Theta$ that is identically equal to zero)
\[
\widehat \rho_t(\mathbf{0};\varLambda) =
\widehat \rho_t \Big(  \set{\rho_t\big(0,\dots,0;  \delta_\theta \circledast P_{t+1,T|\boldsymbol \Theta}\big ),\, \theta \in \boldsymbol \Theta}; P_{t+1,T,\Theta}\Big)
=\rho_t\big(0,\dots,0;   P_{t+1,T}\big) =0.
\]
Thus, $\widehat \rho$ is normalized.

\noindent
2) Suppose $\rho_t$ is normalized and translation invariant and has the support property. Then for any $V\in \bR$, we have
		\[
		\rho_t\big(V,0,\dots,0;  P_{t+1,T}\big) = V + \rho_t\big(0,0,\dots,0;  P_{t+1,T}\big)=V.
		\]
Therefore, for any $U \in  \cZ^{t,t}$ and any $a\in \bR$, by the support, translation invariant, and normalization properties of $\rho_t$, we  have for any $\theta\in\boldsymbol{\Theta}$,
\begin{equation}\label{translation-of-theta}
\rho_t\big(U+a,0,\dots,0;  \delta_\theta \circledast P_{t+1,T|\Theta}\big )
= \rho_t\big(U(\theta)+a,0,\dots,0;  \delta_\theta \circledast P_{t+1,T| \Theta}\big ) = U(\theta)+a.
\end{equation}
Therefore,
\begin{align*}
\widehat \rho_t(U+a;\varLambda)
&= \widehat \rho_t\Big(  \set{\rho_t\big(U+a,0,\dots,0;  \delta_\theta \circledast P_{t+1,T |\Theta}\big ),\, \theta \in \boldsymbol \Theta}; P_{t+1,T,\Theta}\Big)\\
&= \rho_t\big(U+a,0,\dots,0;  P_{t+1,T}\big) = a +  \rho_t\big(U,0,\dots,0;  P_{t+1,T}\big)\\
&=a + \widehat \rho_t\Big(  \set{ \rho_t\big(U,0,\dots,0;  \delta_\theta \circledast P_{t+1,T|\Theta}\big ),\, \theta \in \boldsymbol \Theta}; P_{t+1,T,\Theta}\Big)
= a+ \widehat \rho_t(U;\varLambda).
		\end{align*}
Hence, $\widehat \rho$ is translation invariant.

Similarly, for any $U \in  \cZ^{t,t}$, noting that $\supp_{t,t}(P_{t+1,T})=\supp(\Lambda)$, we deduce
		\begin{align*}
		\widehat \rho_t(U,\varLambda) &= \rho_t\big(U,0,\dots,0;  P_{t+1,T}\big)
		=\rho_t\big(\1_{\text{\rm supp}_{t,t}(P_{t+1,T})}U,0,\dots,0;  P_{t+1,T}\big)\\
		&=\rho_t\big(\1_{\text{\rm supp}(\varLambda)}U,0,\dots,0;  P_{t+1,T}\big)
		=\widehat \rho_t(\1_{\text{\rm supp}(\varLambda)}U,\varLambda).
	\end{align*}
Thus, $\widehat \rho_t$ also has the support property.

\noindent 3) Suppose $\rho_t$ is normalized, translation invariant, and monotonic. Then, for all $U,W \in  \cZ^{t,t}$ such that $U \le W$, employing \eqref{translation-of-theta} with $a=0$, we have
\begin{align*}
\widehat \rho_t(U;\varLambda)  &= \widehat \rho_t\Big(  \set{\rho_t\big(U,0,\dots,0;  \delta_\theta \circledast P_{t+1,T|\Theta}\big ),\, \theta \in \boldsymbol \Theta}; P_{t+1,T,\Theta}\Big)=\rho_t\big(U,0,\dots,0;   P_{t+1,T}\big)\\
&\leq \rho_t\big(W,0,\dots,0;   P_{t+1,T}\big) =
\widehat \rho_t\Big(  \set{\rho_t\big(W,0,\dots,0;  \delta_\theta \circledast P_{t+1,T |\Theta}\big ),\, \theta \in \boldsymbol \Theta}; P_{t+1,T,\Theta}\Big) \\
&= \widehat \rho_t(W;\varLambda),
\end{align*}
and this proves the monotonicity of $\widehat \rho_t$.

\noindent 4) If $\rho$ is monotonic, normalized, parameter consistent, translation invariant, and with support property, then $\widehat \rho_t$ is also normalized, monotonic, normalized,  translation invariant, and has the support property, and in view of Remark~\ref{rem:paramAtT}, $\widehat \rho_t$ is also parameter consistent.

\end{proof}

\subsection{Time Consistency}\label{s:time-consist}

\noindent We now consider the notion of time consistency of risk filters.

\begin{definition}\label{def:wr}
Let $t\in \{1,\ldots,T-1\}$. For any positive measure $\mu_{t+1}$ on $\cX^{T-t}\times \boldsymbol \Theta$ and for any $x\in \mathcal{X}$, we denote by $\mu_{t+1}(\,\cdot\, \| x )$ the measure on $\cX^{T-t-1}\times \boldsymbol \Theta$ given as
\begin{equation}
	\mu_{t+1}(A \times B \| x )=\frac{\mu_{t+1}(\{x\}\times A \times B)}{\mu_{t+1}(\{x\}\times \cX^{T-t-1}\times \boldsymbol \Theta)}.
\end{equation}
Clearly, $\mu_{t+1}(\,\cdot\, \| x )$  is a probability measure.

\end{definition}

In particular, taking $\mu_{t+1}=\delta_\theta \circledast P_{t+1,T|\Theta}$,
\begin{equation}\label{eq:wr}
\delta_\theta \circledast P_{t+1,T|\Theta}(A \times B\, \| x )=\frac{\delta_\theta \circledast P_{t+1,T|\Theta}(\{x\}\times A \times B)}{\delta_\theta \circledast P_{t+1,T|\Theta}(\{x\}\times \cX^{T-t-1}\times \boldsymbol \Theta)}.
\end{equation}
It follows from \eqref{eq:delta} and \eqref{eq:wr} that
\begin{align}\label{eq:wr-delta-0}
\delta_\theta \circledast P_{t+1,T|\Theta}(\,A\times B\, \|\, x )
&=\frac{P_{t+1,T|\theta}(\{x\}\times A)\delta_{ \theta}(B)}{P_{t+1,T|\theta}(\{x\}\times \cX^{T-t-1})}
=\frac{P_{t+1,T|\theta}(\{x\}\times A)}{ P_{t+1,t+1|\theta}(\{x\})}\delta_{ \theta}(B)\\
&=:\widetilde P_{t+1,T|\theta}(\,A\,  \|\, x )\delta_{ \theta}(B), \nonumber
\end{align}
which, in view of \eqref{eq:delta-1} gives
\begin{equation}\label{eq:wr-delta}
\delta_{ \theta} \circledast P^{\pi^{t,h_t}}_{t+1,T|\Theta}(A\times B\| x) =\frac{P^{\pi^{t,h_t}}_{\theta,t+1,T}(\{x\}\times A)}{ P^{\pi^{t,h_t}}_{\theta,t+1,t+1}(\{x\})}\delta_{ \theta}(B)=:\widetilde P^{\pi^{t,h_t}}_{\theta,t+1,T}(\,A\,  \| x )\delta_{ \theta}(B).
\end{equation}

\medskip
The next definition is a version of the dynamic conditional time consistency used in \cite{FanRuszczynski2018}, adapted to the set-up of the present paper.

\begin{definition}\label{sctc-pomdp}
A  dynamic risk filter $\rho=\big\{\rho_{t}\big\}_{t=1,\dots,T}$ is  \emph{time consistent} if for any $t=1,\ldots,T-1$, for any $P_{t+1,T}, Q_{t+1,T} \in \cP_{t+1,T}$,  such that $ P_{t+1,t+1|\Theta}= Q_{t+1,t+1|\Theta}$,  and for any functions\footnote{ The notation $\cdot_{k}$ is a place-holder for $k$ variables.}
$Z_{t,s}(\cdot_{s-t},\cdot_1),W_{t,s}(\cdot_{s-t},\cdot_1)\in {\cZ_{t,s}}$, $s=t+1,\ldots,T$, the inequalities
\begin{align}
\rho_{t+1}&\big(Z_{t,t+1}(x_{t+1},{\cdot_1}),\dots,Z_{t,T}(x_{t+1},\cdot_{T-t-1},{\cdot_1});\delta_\theta \circledast   P_{t+1,T|\Theta}(\,\cdot\, \| x_{t+1} ) \big) \nonumber \\
&\le  \rho_{t+1}\big(W_{t,t+1}(x_{t+1},{\cdot_1}),\dots,W_{t,T}(x_{t+1},\cdot_{T-t-1},{\cdot_1});\delta_\theta  \circledast  Q_{t+1,T|\Theta} (\,\cdot\, \| x_{t+1} ) \big),\label{simple-ineq}\\
& \qquad \qquad \qquad \forall\,\theta \in \boldsymbol \Theta,\quad \forall \,x_{t+1}\in \mathcal{X}, \nonumber
\end{align}
imply that for any function $f_t\in \cZ_{t,t}$
\begin{align}
\rho_{t}&\big(f_t({\cdot_1}),Z_{t,t+1}(\cdot_1,{\cdot_1}),\dots,Z_{t,T}(\cdot_{T-t},{\cdot_1})) ;\delta_\theta \circledast  P_{t+1,T|\Theta})\big) \nonumber\\
& \le \rho_{t}\big(f_t({\cdot_1}),W_{t,t+1}(\cdot_1,{\cdot_1}),\dots,W_{t,T}(\cdot_{T-t},{\cdot_1}));\delta_\theta \circledast  Q_{t+1,T|\Theta})\big),
\quad \forall\,\theta \in \boldsymbol \Theta. \label{consistent-ineq}
\end{align}
\end{definition}

\begin{lemma}\label{l:reduce-to-sigma}
Suppose a dynamic risk filter $\big\{\rho_{t}\big\}_{t=1,\dots,T}$ is normalized, trans\-la\-tion invariant, has the support property and is time consistent.
Let $\theta \in \bTheta$ be fixed. Then the function on $\cZ_{t,t+1} \times \cP(\cX^{T-t}\times \bTheta)$ given as
\begin{equation}\label{eq:rho0-1}
\rho_{t}\big(0,w,0,\dots,0;\delta_\theta \circledast P_{t+1,T| \Theta}\big)
\end{equation}
depends only on the probability $P_{t+1,t+1|\theta}$ and on the function $w(\cdot,\theta)$.
\end{lemma}
\begin{proof}
For any $P_{t+1,T}$ and $x\in \cX$, the support property of $\rho_{t+1}$ implies that
\[
\rho_{t+1}\big(w(x,\cdot),0,\dots,0;\delta_\theta \circledast P_{t+1,T|\Theta}(\cdot \| x)\big)=
\rho_{t+1}\big(w(x,\theta),0,\dots,0;\delta_\theta \circledast P_{t+1,T|\Theta}(\cdot\| x)\big).
\]
Then, by the translation invariance and the normalization properties of $\rho_{t+1}$ we obtain
\[
\rho_{t+1}\big(w(x,\cdot),0,\dots,0;\delta_\theta \circledast P_{t+1,T|\Theta}(\cdot \| x)\big) = w(x,\theta),
\]
which does not depend on $P_{t+2,T|\Theta}$. Hence, for any $Q_{t+1,T}\in \cP_{t+1,T}$
we have
\[
\rho_{t+1}\big(w(x,\cdot),0,\dots,0;\delta_\theta \circledast P_{t+1,T|\Theta}(\,\cdot\, \| x)\big)
= \rho_{t+1}\big(w(x,\cdot),0,\dots,0;\delta_\theta \circledast Q_{t+1,T|\Theta}(\,\cdot\, \| x)\big).
\]
If in addition, $P_{t+1,t+1|\Theta}=Q_{t+1,t+1|\Theta}$, then, by the time consistency,
\[
\rho_t\big(0, w,0,\dots,0;\delta_\theta \circledast P_{t+1,T|\Theta}\big)=
 \rho_t\big(0, w,0,\dots,0;\delta_\theta \circledast Q_{t+1,T|\Theta}\big).
\]
which proves that only the conditional measure $P_{t+1,t+1|\Theta}$  matters in this calculation.
The fact that the knowledge of $w(\cdot,\theta)$ is sufficient, follows from the support property.
This concludes the proof.
\end{proof}

In accordance with the above lemma we define the functions
\[
	\sigma_t: Z^\cX_1 \times \cP(\cX) \to {\bR}, \quad t=1,\ldots, T-1,
\]
as
\begin{equation}\label{sigma-def-tom-original-0}
\sigma_t(v;P_{t+1,t+1|\theta})=\rho_{t}\big(0,w,0,\dots,0;\delta_\theta \circledast P_{t+1,T|\Theta}\big),
\end{equation}
where $w(\cdot,\theta)\equiv v(\cdot)$ and can be arbitrary otherwise. We refer to these functions as \emph{transition risk mappings}.

Note that if $\rho_{t}$ is normalized, monotonic, translation invariant and has support property, then so is $\sigma_t$.

\begin{theorem}	\label{t:filter-structure}
A dynamic risk filter $\rho=\big\{\rho_{t}\big\}_{t=1,\dots,T}$ is normalized, monotonic, translation invariant, has the support property, is parameter consistent, and  time consistent, if and only if the following conditions are satisfied:\\[.3em]
1) Marginal risk mappings $\widehat \rho_t: \cZ_{t,t} \times \cP(\boldsymbol \Theta) \to \bR$, $t\in\cT$, exist, which are normalized, monotonic, translation invariant,	and have the support property;\\[0.3em]
2) Transition risk mappings	given in \eqref{sigma-def-tom-original-0} are such that

\begin{enumerate}[(i)]
\item
For all $t=1, \ldots, T-1$, $\sigma_t(\cdot;\cdot)$ is  normalized, monotonic, translation invariant, and has the support property;

\item
For any $P_{t+1,T} \in \mathcal{P}_{t+1,T}$, $t=1, \dots ,T-1$, and for any functions $Z_{t,s}\in \cZ_{t,s}$, $s\in\cT_t$, we have that\footnote{Recall Definition \ref{def:wr}.
The notation $\sigma_t\big(\rho_{t+1}\big(Z_{t,t+1}(\diamond,{\theta}), \dots ,Z_{t,T}(\diamond,\cdot_{T-t-1},{\theta});\delta_{{\theta}}  \circledast P_{t+1,T| \Theta}(\,\cdot\, \| \diamond )\big);P_{t+1,t+1|\theta}\big)$, where we use $\diamond$ as place holder, means that $P_{t+1,t+1|\theta}$ acts on $w(x, \theta)= \rho_{t+1}\big(Z_{t,t+1}(x,{\theta}), \dots ,f_{t,T}(x,\cdot_{T-t-1},{\theta});\break\delta_{{\theta}}  \circledast P_{t+1,T|\Theta}(\,\cdot\, \| x )\big)$ as a function of $x$. }
\begin{multline}\label{li-risk-trans-map-controlled}
\rho_{t}(Z_{t,t},Z_{t,t+1},\dots,Z_{t,T}; P_{t+1,T}) =
\widehat \rho_t\Big(\Big\{Z_{t,t}(\theta)+ \sigma_t\big(
\rho_{t+1}\big(Z_{t,t+1}(\diamond,\cdot_1), \dots ,Z_{t,T}(\diamond,\cdot_{T-t-1},\cdot_1);\\ \delta_{{\theta}}  \circledast P_{t+1,T|\Theta}(\,\cdot\, \| \diamond )\big);P_{t+1,t+1|\theta}\big),\theta \in \bTheta\Big\};
P_{t+1,T,\Theta} \Big).
\end{multline}
For any function $Z_{T,T}\in \cZ_{T,T}$ and $P_{T+1,T}\in \cP(\boldsymbol{\Theta})$
\begin{equation}\label{final-conditional}
\rho_{T}(Z_{T,T}; P_{T+1,T}) = \widehat \rho_T\Big(Z_{T,T},P_{T+1,T}\Big).
\end{equation}
\end{enumerate}
\end{theorem}

\begin{proof} We fix $P_{t+1,T} \in \mathcal{P}_{t+1,T}$ and $Z_{t,s}\in \cZ_{t,s}$, $s=t,\ldots,T$. Since $t$ is fixed, to alleviate the notations we simply write $Z_s$, instead of $Z_{t,s}$ in this proof.

Since the risk filter is parameter consistent, Theorem \ref{t:disintegration} yields the existence of mappings $\widehat \rho_t$ such that
\[
\rho_{t}(Z_{t},Z_{t+1},\dots,Z_{T}; P_{t+1,T}) =
\widehat \rho_t\Big(\Big\{\rho_{t}\big(Z_{t},Z_{t+1},\dots,Z_{T};\delta_\theta \circledast P_{t+1,T| \Theta}\big),\theta \in \bTheta\Big\};
P_{t+1,T,\Theta} \Big).
\]

It follows from Proposition~\ref{prop:disintegration} that $\widehat \rho$ is normalized, monotonic, trans\-la\-tion invariant, has the support property.

Next, we derive an equivalent expression for the first argument of $\widehat \rho_t$ that will prove \eqref{li-risk-trans-map-controlled}.
Define the function
\[
w(x,\theta) = \rho_{t+1}\big(Z_{t+1}(x,{\cdot_1}),\dots,Z_{T}(x,\cdot_{T-t-1},{\cdot_1});\delta_\theta \circledast P_{t+1,T|\boldsymbol \Theta}(\,\cdot\, \| x)\big),
\quad x\in \cX,\quad \theta\in \boldsymbol \Theta.
\]
Then, for any fixed $x\in \cX$ and $\theta\in\boldsymbol{\Theta}$, we use the support, translation invariance, and the normalization properties in the chain of equations below: 		
\begin{align*}
	\rho_{t+1}&\big(w(x,\cdot),0,\dots,0;\delta_\theta \circledast P_{t+1,T|\Theta}(\,\cdot\, \| x)\big) \\
&=\rho_{t+1}\big(w(x,\theta),0,\dots,0;\delta_\theta \circledast P_{t+1,T|\Theta}(\,\cdot\, \| x)\big)\\
&=w(x,\theta) + \rho_{t+1}\big(0,0,\dots,0;\delta_\theta \circledast P_{t+1,T|\Theta}(\,\cdot\, \| x)\big)
 = w(x,\theta) \\
&= \rho_{t+1}\big(Z_{t+1}(x,{\cdot_1}),\dots,Z_{T}(x,\cdot_{T-t-1},{\cdot_1});\delta_\theta \circledast P_{t+1,T|\Theta}(\,\cdot\, \| x)\big).
\end{align*}
In view of the assumed time consistency of $\rho$, the above implies that  for every $Z_{t,t}\in \cZ_{t,t}$,
\[
\rho_{t}\big(Z_{t},Z_{t+1},\dots,Z_{T};\delta_\theta \circledast P_{t+1,T|\Theta}\big)
= \rho_{t}\big(Z_{t},w,0,\dots,0;\delta_\theta \circledast P_{t+1,T|\Theta}\big)=:I_1.
\]
Thus, by using the translation invariance and the support properties again, we conclude that,
 for all $\theta\in {\boldsymbol \Theta}$,
\begin{equation*}
I_1= \rho_{t}\big(Z_{t}(\theta),w,0,\dots,0;\delta_\theta \circledast P_{t+1,T|\Theta}\big)
=  Z_{t}(\theta) + \rho_{t}\big(0,w,0,\dots,0;\delta_\theta \circledast P_{t+1,T|\Theta}\big).
\end{equation*}
Using \eqref{sigma-def-tom-original-0}, we get
\[
I_1 = Z_{t}(\theta) + \sigma_t\big(w(\cdot,\theta),P_{ t+1,t+1|\theta}\big).
\]
Finally, from here, noting that by support property
\begin{equation}\label{eq:wxtheta}
w(x,\theta) = \rho_{t+1}\big(Z_{t+1}(x,\theta),\dots,Z_{T}(x,\cdot_{T-t-1},\theta);\delta_\theta \circledast P_{t+1,T|\Theta}(\,\cdot\, \| x)\big),
\quad x\in \cX,\quad \theta\in \boldsymbol \Theta,
\end{equation}
we obtain the representation \eqref{li-risk-trans-map-controlled}. The representation \eqref{final-conditional} follows from the definition of $\rho_{T}$ and the form of $\widehat \rho_T$.

Next we prove the converse statement by backward induction in time. For $t=T$, the conditional risk filter \eqref{final-conditional} has all the postulated properties, with the exception of the time consistency, because $\widehat \rho_T$ does (see Proposition~\ref{prop:disintegration}).

Suppose the conditional risk filters $\rho_{s}$, $s=t+1,\dots,T$ are normalized, monotonic, translation invariant, have the support property, are parameter consistent, and  time consistent. We will verify these properties for $\rho_t$ given by formula \eqref{li-risk-trans-map-controlled}.
The translation invariance follows from the translation invariance of  $\widehat \rho_t$. The normalization and the monotonicity follow immediately from the normalization and the monotonicity  of $\sigma_t$, $\rho_{t+1}$ and $\widehat \rho_t$.

We now verify the support property. For every $\theta$, and $x$ define $\mu_{\theta,x}(\cdot)=\delta_{\theta}  \circledast P_{t+1,T|\theta}(\cdot \| x )$, $A(\theta) = \supp(P_{t+1,t+1|\theta})\subset\cX$, $B=\supp(P_{t+1,T, \Theta})\subset \boldsymbol{\Theta}$. Then, by \eqref{li-risk-trans-map-controlled}, \eqref{eq:wxtheta}, the support property of $\hat\rho_t$ and $\sigma_t$, and by Remark~\ref{rem:locality} applied to $\sigma_t$, we deduce that
\begin{align}
\rho_{t}(Z_{t}, & Z_{t+1}, \ldots, Z_{T}; P_{t+1,T})  =
	\hat\rho_t\Big(	\Big\{ Z_{t}(\theta) + \sigma_t(w(\diamond,\theta);P_{t+1,t+1,|\theta});
	\theta\in\boldsymbol{\Theta}\Big\}; P_{t+1,T,\Theta}\Big) \nonumber\\
	& = \hat\rho_t\Big(\Big\{ \1_B(\theta)Z_{t}(\theta) + \1_B(\theta)\sigma_t(w(\diamond,\theta);P_{t+1,t+1,|\theta}); \theta\in\boldsymbol{\Theta}\Big\}; P_{t+1,T,\Theta}\Big)\nonumber  \\
	& = \hat\rho_t\Big(\Big\{ \1_B(\theta)Z_{t}(\theta) + \sigma_t(\1_B(\theta)\1_{A(\theta)}(\diamond) w(\diamond,\theta);P_{t+1,t+1,|\theta}); \theta\in\boldsymbol{\Theta}\Big\}; P_{t+1,T,\Theta}\Big).  \label{eq:temp2}
\end{align}
By the assumed support property of $\rho_{t+1}$, and in view of Remark~\ref{rem:locality} applied to $\rho_{t+1}$, we obtain, using \eqref{eq:wxtheta} again,
\begin{multline*}
\1_{A(\theta)}(x)\1_B(\theta) w(x,\theta)
= \rho_{t+1}\big(\1_{A(\theta)}(x)\1_B(\theta)Z_{t+1}(x,\theta),\1_{A(\theta)}(x)\1_B(\theta)\1_{\supp_{t+2}(\mu_{\theta,x})} Z_{t+2}(x,\cdot,\theta),\\
\dots,\1_{A(\theta)}(x)\1_B(\theta)\1_{\supp_T(\mu_{\theta,x})}Z_{T}(x,\cdot_{T-t-1},\theta);\mu_{\theta,x}\big),
\end{multline*}
for every $x\in\cX$ and $\theta\in\boldsymbol{\Theta}$. From here and \eqref{eq:temp2}, combined with the normalization property of $\rho_{t+1}$, and the fact that $\1_{A(\theta)}(x)\1_B(\theta)\1_{\supp_{s}(\mu_{\theta,x})} \leq 1_{\supp_s(P_{t+1,T})}, \ s=t,\ldots,T$, we obtain the support property of~$\rho_{t}$.

Next we prove the parameter consistency. Assume that \eqref{eq:paramCon1} is satisfied for a fixed $\bar \theta\in\boldsymbol{\Theta}$, and denote by $\bar P_{t+1,T}= \delta_{\bar \theta}\circledast P_{t+1,T|\Theta}$ and $\bar Q_{t+1,T}= \delta_{\bar \theta}\circledast Q_{t+1,T|\Theta}$. We note that\footnote{We use the convention that $\frac{0}{0}=0$ when considering  $
\bar P_{t+1,t+1|\theta}$ and $\bar P^x_{t+1,T|\theta} (\cdot \| x)$.}
$$
\bar P_{t+1,t+1|\theta} = P_{t+1,t+1|\bar \theta} \cdot \1_{\bar \theta}(\theta), \quad \bar P_{t+1,T,\Theta} = \delta_{\bar \theta}, \quad
\bar P_{t+1,T|\theta} (\cdot \| x) = P_{t+1,T|\bar \theta}(\cdot \| x)\1_{\bar \theta}(\theta).
$$

Using this, and in view of \eqref{li-risk-trans-map-controlled}, we can write \eqref{eq:paramCon1} as follows (with measures $\bar P_{t+1,T}$ and $\bar Q_{t+1,T}$ in place of $P_{t+1,T}$ and $Q_{t+1,T}$):
\begin{multline*}
\widehat \rho_t\Big(\Big\{Z_{t}(\theta)+ \sigma_t\big(
\rho_{t+1}\big(Z_{t+1}(\diamond,{\theta}), \dots ,Z_{T}(\diamond,\cdot_{T-t-1},{\theta});\\
\delta_{{\bar \theta}}(\cdot) P_{t+1,T|\bar \theta}(\,\cdot\, \| \diamond )\1_{\bar \theta}(\theta);P_{t+1,t+1|\bar \theta}\1_{\bar \theta}(\theta)\big),\theta \in \bTheta\Big\};
\delta_{{\bar \theta}}\Big)\\
\leq
 \widehat \rho_t\Big(\Big\{W_{t}(\theta)+ \sigma_t\big(
\rho_{t+1}\big(W_{t+1}(\diamond,{\theta}), \dots ,W_{T}(\diamond,\cdot_{T-t-1},{\theta});\\
\delta_{{\bar \theta}}(\cdot) Q_{t+1,T|\bar \theta}(\,\cdot\, \| \diamond )\1_{\bar \theta}(\theta);Q_{t+1,t+1|\bar \theta}\1_{\bar \theta}(\theta)\big),\theta \in \bTheta\Big\};
\delta_{{\bar \theta}}\Big).
\end{multline*}
By the support property of $\hat\rho_t$ and $\sigma_t$, and by the normalization and monotonicity of $\hat\rho_t$, we obtain that
\begin{align*}
Z_{t}(\bar \theta)&+ \sigma_t\big(
\rho_{t+1}\big(Z_{t+1}(\diamond,{\bar \theta}), \dots ,Z_{T}(\diamond,\cdot_{T-t-1},{\bar \theta});\delta_{{\bar \theta}}(\cdot) P_{t+1,T|\bar \theta}(\,\cdot\, \| \diamond );P_{t+1,t+1|\bar \theta}\big )\\
\leq& W_{t}(\bar \theta)+ \sigma_t\big(
\rho_{t+1}\big(W_{t+1}(\diamond,{\bar \theta}), \dots ,W_{T}(\diamond,\cdot_{T-t-1},{\bar \theta});\delta_{{\bar \theta}}(\cdot) Q_{t+1,T|\bar \theta}(\,\cdot\, \| \diamond );Q_{t+1,t+1|\bar \theta}\big ),
\end{align*}
for any $\bar \theta\in\boldsymbol{\Theta}$. From here, applying $\hat\rho_t$ to both sides, since we assumed that $P_{t+1,T,\boldsymbol \Theta}=Q_{t+1,T,\boldsymbol \Theta}$, employing monotonicity of $\hat\rho_t$, we obtain \eqref{eq:paramCon2}, and the parameter consistency of $\rho_{t}$ is proved.

Finally, let us verify the time consistency at time $t$. If the inequalities  \eqref{simple-ineq} are satisfied, then it follows from
 the monotonicity of $\sigma_t$ with respect to its first argument that for all $\theta  \in \bTheta$
\begin{multline*}
G_1(\theta):= \sigma_t\big(
\rho_{t+1}\big(Z_{t+1}(\diamond,\theta),\dots,Z_{T}(\diamond,\cdot_{T-t-1},\theta);\delta_{\theta}  \circledast P_{t+1,T|\Theta}(\cdot \| \diamond)\big);P_{t+1, t+1|\theta} \big)\\
\le \sigma_t\big(
\rho_{t+1}\big(W_{t+1}(\diamond,\theta),\dots,W_{T}(\diamond,\cdot_{T-t-1},\theta);\delta_{\theta}  \circledast Q_{t+1,T|\Theta}(\cdot \| \diamond)\big);P_{t+1, t+1|\theta}\big)
=:G_2(\theta).
\end{multline*}
Then, from the monotonicity of $\widehat \rho_t$  we get, for any function $f_t\in \mathcal{Z}_{t,t}$,
$$
\hat\rho_t(f_t+G_1; \delta_\theta) \leq \hat\rho_t(f_t+G_2;\delta_\theta), \quad \theta\in\boldsymbol{\Theta}.
$$
From here, using the support property of $\widehat \rho_t, \sigma_t$ and $\rho_{t+1}$, along \eqref{li-risk-trans-map-controlled}, we obtain \eqref{consistent-ineq}, and thus time consistency at $t$ is verified.

By induction, all properties hold true for $t=1,\dots,T$, and the proof is complete.
\end{proof}

\begin{example}\label{ex:additive-reward}

We consider a very special conditional risk filter, given as the expectation of an additive functional under the measure $P_{t+1,T}$. Specifically, we let
\begin{align*}
\rho_{t}(Z_{t,t},\ldots,Z_{t,T};P_{t+1,T}) & =
\int_{\mathcal{X}^{T-t}\times \boldsymbol \Theta}
\sum_{k=t}^T Z_{t,k}(x_{t+1},\ldots,x_{k}, \theta)  P_{t+1,T}(dx_{t+1},\cdots,dx_T,d\theta) \\
= & E_{P_{t+1,T}}\sum_{k=t}^T Z_{t,k}.
\end{align*}

Clearly, this $\rho_{t}$ is normalized, monotonic, translation invariant, and has the support property (cf.  Definition \ref{basic-prop-pomdp}).

Next, note that for  this $\rho_{t}$ the inequality \eqref{eq:paramCon1} becomes (cf. \eqref{eq:delta})
\begin{align*}
\int_{\mathcal{X}^{T-t}} \sum_{k=t}^T Z_{t,k}(x_{t+1:k},\theta)  P_{t+1,T|\theta}(dx_{t+1},\cdots,dx_T) \leq
\int_{\mathcal{X}^{T-t}} \sum_{k=t}^T Z_{t,k}(x_{t+1:k}, \theta)  Q_{t+1,T|\theta}(dx_{t+1},\cdots,dx_T),
\end{align*}
for any $\theta\in\boldsymbol{\Theta}$. Assuming that $P_{t+1,T,\Theta}=Q_{t+1,T, \Theta}$, multiplying the last inequality by $P_{t+1,T,\Theta}(\theta)$, and summing up with respect to  $\theta\in\boldsymbol{\Theta}$, the inequality \eqref{eq:paramCon2} follows, and hence the parameter consistency is true.

The time consistency follows by similar arguments. Indeed, \eqref{simple-ineq} becomes (cf. \eqref{eq:wr-delta-0})
\begin{align*}
\int_{\mathcal{X}^{T-t-1}} &\sum_{k=t+1}^T Z_{t,k}(x_{t+1},x_{t+2:k},\theta)  \;\widetilde P_{t+1,T|\theta}(dx_{t+2},\cdots,dx_T \| x_{t+1}) \\
& \leq  \int_{\mathcal{X}^{T-t-1}} \sum_{k=t+1}^T Z_{t,k}(x_{t+1},x_{t+2:k}, \theta)\; \widetilde Q_{t+1,T|\theta}(dx_{t+2},\cdots,dx_T \| x_{t+1}),
\end{align*}
for any $x_{t+1}\in\cX$ and $\theta\in\boldsymbol{\Theta}$.
Assuming that $P_{t+1,t+1|\Theta} = Q_{t+1,t+1|\Theta}$, multiplying both parts by $P_{t+1,t+1|\theta}(x_{t+1})$,  and noting that (cf. \eqref{eq:wr-delta-0})
$$
\widetilde P_{t+1,T | \theta}(\cdot \| x_{t+1}) P_{t+1,t+1|\theta}(x_{t+1}) = P_{t+1,T|\theta}(x_{t+1},\cdot),
$$
for any function $f_t\in \cZ_{t,t}$ we have (cf. \eqref{eq:wr-delta-0})
\begin{align*}
f_t(\theta) + \int_{\mathcal{X}^{T-t-1}} &\sum_{k=t+1}^T Z_{t,k}(x_{t+1},x_{t+2:k},\theta) \;P_{t+1,T|\theta,}(\{x_{t+1}\},dx_{t+2},\cdots,dx_T ) \\
& \leq  f_t(\theta)+\int_{\mathcal{X}^{T-t-1}} \sum_{k=t+1}^T Z_{t,k}(x_{t+1},x_{t+2:k}, \theta)  \;Q_{t+1,T|\theta}(\{x_{t+1}\},dx_{t+2},\cdots,dx_T ),
\end{align*}
After summing up with respect to $x_{t+1}$ we obtain \eqref{consistent-ineq}, and thus the time consistency is proved.

\smallskip
We complete this example by observing that in the this case we have that $\widehat\rho_t(f, {P}') = E_{{P}'}(f)$, for $f\in Z_{t,t}$ and for  ${P}'\in\cP(\boldsymbol{\Theta})$, and that and $\sigma_{t}(v; {P}'') = E_{{P}''}(v)$, for $v \in \mathcal{Z}^{\cX}_1$ and  ${P}''\in\cP(\cX)$.
\end{example}
	
\begin{example}\label{rem:notations}
Let us cast  Example \ref{ex:additive-reward} in the setup of Section~\ref{s:pomdp}. For this, we fix a history $h_t=(x_1,\ldots,x_t)\in \cH_t$ and $\pi\in \Pi$, and we take
\[
	 Z_{t,t}(\theta):=Z^{\pi,h_t}_{\theta,t,t} =c_t(x_t,\pi_t(h_t),\theta),\]  \[Z_{t,s}(x_{t+1},\ldots,x_s,\theta):=Z^{\pi,h_t,x_{t+1},\ldots,x_s}_{\theta,t,s} =c_s(x_s,\pi^{t,h_t}_s(x_{t+1},\ldots,x_s),\theta), s=t+1,\ldots,T, \]
and
\[
P_{t+1,T} = P_{t+1,T}^{\pi^{t,h_t}}.
\]
The conditional risk filter of Example \ref{ex:additive-reward} becomes a conditional expectation (cf. Lemma \ref{lemma:2.2})
\begin{align}\label{mult}
&\rho_{t}\Big(c_t(x_t,\pi_t(h_t),\cdot),c_{t+1}(\cdot,\pi_{t+1}(h_t,\cdot),\cdot),\cdots ,c_{T}(\cdot,\pi_{T}(h_t,\cdot,\ldots,\cdot),\cdot),P^{\pi^{t,h_t}}_{t+1,T}\Big )\nonumber  \\
&= E^\pi\left[c_t(x_t,\pi_t(h_t),\Theta) + \sum_{s=t+1}^T c_s(\wh X_s, \pi_s^{t,h_t}(\wh X_{t+1},\ldots,\wh X_s), \Theta) \mid \wh H_t=h_t\right],
\end{align}
for $t=1,\ldots,T$, where we use the standard convention that an empty sum is zero (i.e. $\sum_{s={T+1}}^T\cdots =0$ in our case).

In view of \eqref{eq:I-T-integrated2} we also have
\begin{align}\label{mult-1}
&\rho_{t}\Big(c_t(x_t,\pi_t(h_t),\cdot),c_{t+1}(\cdot,\pi_{t+1}(h_t,\cdot),\cdot),\cdots ,c_{T}(\cdot,\pi_{T}(h_t,\cdot,\ldots,\cdot),\cdot),P^{\pi^{t,h_t}}_{t+1,T}\Big )
\nonumber \\
&=  \widehat{\rho}_t\Big(
			\Big\{c_t(x_t,\pi_t(h_t),\theta)+ \sigma_t \Big (\rho_{t+1}\Big(c_{t+1}(\diamond,\pi_{t+1}(h_t,\diamond),\cdot),c_{t+2}(\cdot,\pi_{t+2}(h_t,\diamond,\cdot),\cdot), \nonumber
\\ & \cdots ,c_{T}(\cdot,\pi_{T}(h_t,\diamond,\cdot,\ldots,\cdot),\cdot);\delta_{{\theta}}  \circledast P^{\pi^{t,h_t}}_{t+1,T|\Theta}(\,\cdot\, \| \diamond )\Big );P^{\pi^{t,h_t}}_{t+1,t+1|\theta}\Big),\theta \in \bTheta\Big \};   \xi^{\pi,h_t}_{t}  \Big )
 \nonumber \\
&=  \widehat{\rho}_t\Big(
			\Big\{c_t(x_t,\pi_t(h_t),\theta)+ \sigma_t \Big (\rho_{t+1}\Big(c_{t+1}(\diamond,\pi_{t+1}(h_t,\diamond),\cdot),c_{t+2}(\cdot,\pi_{t+2}(h_t,\diamond,\cdot),\cdot), \nonumber
\nonumber \\ & \cdots ,c_{T}(\cdot,\pi_{T}(h_t,\diamond,\cdot,\ldots,\cdot),\cdot);P^{\pi^{t,h_t}}_{\theta,t+1,T}(\{\diamond\}\times \cdot)\delta_{ \theta}(\cdot)\Big );P^{\pi^{t,h_t}}_{\theta, t+1}\Big),\theta \in \bTheta\Big \};   \xi^{\pi,h_t}_{t}  \Big ),
\end{align}
	 where in the last equality we used \eqref{three-eqs} and \eqref{eq:wr-delta}, and where, for a function $f$ on $\bTheta$ and a measure $\xi$ on $\bTheta$,
\begin{align}
			\label{eq:main1}\widehat \rho _t\Big (\big \{ f(\theta),\theta \in \bTheta\big \};\xi^{\pi,h_t}_{t} \Big )= \widehat \rho _t\big (f;\xi^{\pi,h_t}_{t} \big )=\int_\bTheta f(\theta)\;\xi^{\pi,h_t}_{t} (d\theta)=E_{\xi^{\pi,h_t}_{t} } (f),
\end{align}	
and	where, for a function $v$ on $\cX$, we have (cf. \eqref{three-eqs}, \eqref{sigma-def-tom-original-0}, and \eqref{eq:delta-1})
\begin{align}
			\label{eq:main2}
\sigma_t\left (v,P^{\pi^{t,h_t}}_{\theta,t+1}\right )=\int_{\mathcal{X}}v(x)P^{\pi^{t,h_t}}_{\theta,t+1}d(x)=E_{P^{\pi^{t,h_t}}_{\theta,t+1}}(v).
\end{align}

\end{example}

\begin{example}\label{e:trials-tom}

In the previous example we proceeded from $\rho$ to $\sigma$ (via $\widehat \rho$). Here, we will do the opposite.

In clinical trials, the potency of a drug is characterized by an unknown parameter ${\theta}$.
The purpose of the trials is to estimate ${\theta}$ and to determine the optimal dose.
Let us assume for simplicity that ${\theta}$ \emph{is} the optimal dose.
If a dose $u_1$ is administered to a patient, a response $X_2$ is observed (the subscript 2 indicates that $X_2$ is not known when $u_1$ is determined).
$X_2$ is a Bernoulli random variable, with $X_2=1$ representing toxic response, and $X_2=0$  nontoxic. The probability of toxic response
is a function of $\theta$ and $u_1$, that is, $P[X_2=1] = \Psi(\theta,u_1)$. The ``cost'' is $c(\theta,u_1)$; it depends on both the applied and best doses. The cost is not observed; we only know whether the patient was toxic or not. In the second stage, the dose $u_2$ is administered to the next patient, the patient's response $X_3$ observed, and cost $c(\theta,u_2)$ incurred. The process continues for $T$ stages, with $u_T$ being the final dose recommendation, whose cost is equal to $c(\theta,u_T)$. For example,
the cost may have the form $c(\theta,u) = |u-\theta|$ to penalize for the over- and under-dosage. It is never observed.

The problem can be cast to our setting. The state space $\cX$ is $\{0,1\}$, while the unknown parameter space $\boldsymbol \Theta$ is an interval of the real line or a finite subset of the real line. Given the set-up adopted in this paper, we assume that $\boldsymbol \Theta$ is a finite subset of the real line. The transition kernel does not depend on $X$ at all; the distribution of the next $X_{t+1}$ depends on $\theta$ and $u$:
 \[
 K_\theta(0|x,u) = 1- \Psi(\theta,u),\quad K_\theta(1|x,u) = \Psi(\theta,u).
 \]
Thus, we have
 \[
 P^{\pi^{t,h_t}}_{\theta,t+1}(y) =  K_\theta(y|x_t,\pi_t(h_t)) = \1_{y=0}(1- \Psi(\theta,\pi_t(h_t)))+\1_{y=1}\Psi(\theta,\pi_t(h_t)),\quad y\in\{0,1\}.
 \]
There is a considerable leverage in choosing the form of $\sigma_t$ in a way consistent with the above set-up. For example, one can choose $\sigma_t$ in terms of the entropic risk measures, as follows
$$
\sigma(w,P) = \frac{1}{\kappa}\ln\int_\cX e^{\kappa w(y)}P(dy),
$$
for a function $f$ on $\cX$, $P\in\cP(\cX)$ and a constant $\kappa>0$ . Consequently,  for $t=1,\ldots,T-1$, using \eqref{eq:new-measure} and \eqref{ii-1}  we obtain
\begin{align*}
\sigma_t(w(\cdot,\theta);P^{\pi^{t,h_t}}_{t+1,t+1|\theta})
&=\frac{1}{\varkappa} \ln\bigg(\big(1- \Psi(\theta,\pi_t(h_t))\big) e^{\varkappa w(0,\theta)} + \Psi(\theta,\pi_t(h_t))e^{\varkappa w(1,\theta)}\bigg)\\
&=\frac{1}{\varkappa} \ln\int_{\cX}e^{\varkappa w(y,\theta)} P^{\pi^{t,h_t}}_{\theta,t+1}(dy),
\end{align*}
with $\sigma_T=0$.

Now, for a function $f$ on $\bTheta$ and a measure $\xi\in\cP(\bTheta)$, let
\[\widehat \rho _t\Big (\big \{ f(\theta),\theta \in \bTheta\big \};\xi\Big )= \widehat \rho _t (f;\xi)
=\frac{1}{\varkappa}\ln\int_\bTheta e^{\varkappa f(\theta)}\xi(d\theta), \qquad t\in\cT.
\]	
Given the above, we obtain for $t=T$ 
\begin{equation}\label{key1}
\widehat \rho _T\Big (\big \{ v(\theta),\theta \in \bTheta\big \}; \xi^{\pi,h_T}_{T} \Big )
=\frac{1}{\varkappa}\ln \int_\bTheta e^{\varkappa v(\theta) }\xi^{\pi,h_T}_{T}(d\theta),
\end{equation}
and for $t=1,\dots,T-1$,
\begin{align}\label{key}
\widehat \rho _t\Big (\big \{ v(\theta)+\sigma_t(w(\cdot,\theta);P^{\pi^{t,h_t}}_{t+1,t+1|\theta}),\theta \in \bTheta\big \}; \xi^{\pi,h_t}_{t} \Big )
&=\frac{1}{\varkappa}\ln \int_\bTheta \int_{\cX} e^{\varkappa(v(\theta) + w(y,\theta))}P^{\pi^{t,h_t}}_{\theta,t+1}(dy)\xi^{\pi,h_t}_{t}(d\theta) \nonumber \\
&=\frac{1}{\varkappa}\ln E^\pi [e^{\varkappa(v(\Theta) + w(X_{t+1},\Theta))}|\widehat H_t=h_t],
\end{align}
where the last equality follows from Lemma \ref{cacy-cacy}.

We will now derive a generic formula for $\rho_{t}$, generated by \eqref{li-risk-trans-map-controlled} and  $\sigma_t$ and $\widehat \rho _t$ as above, in case of the generic cost functions as in \eqref{eq:cost1} and \eqref{eq:cost2}. Let us fix an admissible  strategy $\pi$. For $t=T$ we have
\begin{align*}
\rho_{T}(c_{T}(x_T,\pi_T(h_T),\cdot),P^{\pi^{T,h_T}}_{T+1,T}) & =\widehat \rho_T(\{c_{T}(x_T,\pi_T(h_T),\theta),\theta \in \bTheta\};P^{\pi^{T,h_T}}_{T+1,T, \Theta})\\
& = \widehat \rho_T(c_{T}(x_T,\pi_T(h_T),\cdot);P^{\pi^{T,h_T}}_{T+1,T,\Theta})=\widehat \rho_T(c_{T}(x_T,\pi_T(h_T),\cdot);\xi^{\pi,h_T}_{T})\\
& = \frac{1}{\varkappa}\ln\int_\bTheta e^{\varkappa c_{T}(x_T,\pi_T(h_T),\theta)}\xi^{\pi,h_T}_{T}(d\theta) \\
&= \frac{1}{\varkappa}\ln E^\pi(e^{\varkappa c_{T}(x_T,\pi_T(h_T),\Theta)}|\widehat H_T=h_T).
\end{align*}
Now, note that
\[
\rho_{T}(c_{T}(x_T,\pi_T(h_T),\cdot),\delta_{\theta})= c_{T}(x_T,\pi_T(h_T),\theta),
\]
and thus
\begin{align*}
\sigma_{T-1}\big(\rho_{T}\big(c_{T}(\diamond,\pi_T(h_{T-1},\diamond),\theta);
\delta_{{\theta}}); & P^{\pi^{T-1,h_{T-1}}}_{T,T|\theta}\big)=\sigma_{T-1}\big(c_{T}(\diamond,\pi_T(h_{T-1},\diamond),\theta) ;P^{\pi^{T-1,h_{T-1}}}_{T,T|\theta}\big)
\\
& =\frac{1}{\varkappa} \ln\int_{\cX}e^{\varkappa c_{T}(x_T,\pi_T(h_{T-1},x_T),
\theta)} P^{\pi^{T-1,h_{T-1}}}_{\theta,T}(dx_T).
\end{align*}
So, for $t=T-1,$  we have
\begin{align*}
\rho_{T-1}(&c_{T-1}(x_{T-1},\pi_{T-1}(h_{T-1}),\cdot),c_{T}(\cdot,\pi_T(h_{T-1},\cdot),\cdot),P^{\pi^{T-1,h_{T-1}}}_{T,T}) \\
&=\widehat \rho_{T-1}\Big(	\Big\{c_{T-1}(x_{T-1},\pi_{T-1}(h_{T-1}),\theta)+ \sigma_{T-1}\big(\rho_{T}\big(c_{T}(\diamond,\pi_T(h_{T-1},\diamond),\cdot);
\delta_{{\theta}});P^{\pi^{T-1,h_{T-1}}}_{T,T|\theta} \big)\big), \\
& \qquad \qquad \qquad \qquad \qquad \qquad \theta \in \bTheta\Big\};P^{\pi^{T-1,h_{T-1}}}_{T,T,\Theta} \Big)
\end{align*}
\begin{align*}
\qquad & =\widehat \rho_{T-1}\Big(
				\Big\{c_{T-1}(x_{T-1},\pi_{T-1}(h_{T-1}),\theta)+ \sigma_{T-1}\big(\rho_{T}\big(c_{T}(\diamond,\pi_T(h_{T-1},\diamond),\cdot);
  \delta_{{\theta}});P^{\pi^{T-1,h_{T-1}}}_{T,T|\theta} \big)\big),\\
& \qquad \qquad \qquad \qquad \qquad \qquad \theta \in \bTheta\Big\};\xi^{\pi,h_{T-1}}_{T-1}\Big )\\
&=\frac{1}{\varkappa}\ln \int_\bTheta \int_{\cX} e^{\varkappa(c_{T-1}(x_{T-1},\pi_{T-1}(h_{T-1}),\theta) + c_{T}(x_T,\pi_{T}(h_{T-1},x_T),\theta))}
P^{\pi^{T-1,h_{T-1}}}_{\theta,T}(dx_T)\xi^{\pi,h_{T-1}}_{T-1}(d\theta) \\
&= \frac{1}{\varkappa}\ln E^\pi(e^{\varkappa(c_{T-1}(x_{T-1},\pi_{T-1}(h_{T-1}), \Theta) + c_{T}(X_T,\pi_{T}(h_{T-1},X_T), \Theta))}|\widehat H_{T-1}=h_{T-1}),
\end{align*}
where we used \eqref{eq:new-measure} and \eqref{ii-1} for the second to the last equality, and where the last equality follows from Lemma~\ref{cacy-cacy}.

Next, note that
\begin{align*}
\rho_{T-1}\big(c_{T-1}(&\diamond,\pi_{T-1}(h_{T-2},\diamond),\cdot), c_{T}(\cdot,\pi_{T}(h_{T-2},\diamond,\cdot),\cdot);
\delta_{\theta}\circledast P^{\pi^{T-1,(h_{T-2},\diamond)}}_{T,T|\Theta}(\cdot \| \diamond)\big ) \\
 & =\frac{1}{\varkappa}\ln \int_{\cX} e^{\varkappa(c_{T-1}(\diamond,\pi_{T-1}(h_{T-2},\diamond),\theta) + c_{T}(y,\pi_{T}(h_{T-2},\diamond,y),\theta))}
P^{\pi^{T-1,(h_{T-2},\diamond)}}_{\theta,T}(dy),
\end{align*}
and
\begin{align*}
\sigma_{T-2}\big(\rho_{T-1}&\big(c_{T-1}(\diamond,\pi_{T-1}(h_{T-2},\diamond),\theta), c_{T}(\cdot,\pi_{T}(h_{T-2},\diamond,\cdot),\theta); \\
& \qquad \qquad \qquad \qquad \qquad \qquad  \delta_{{\theta}}\circledast P^{\pi^{T-1,(h_{T-2},\diamond)}}_{T,T|\Theta}(\cdot \| \diamond) \big);P^{\pi^{T-2,h_{T-2}}}_{T-1,T-1|\theta}\big)\\
&=\frac{1}{\varkappa}\ln \int_{\cX} e^{\varkappa(c_{T-1}(x_{T-1},\pi_{T-1}(h_{T-2},x_{T-1}),\theta) + c_{T}(x_{T},\pi_{T}(h_{T-2},x_{T-1},x_{T}),\theta))}\\
& \qquad \qquad \qquad \qquad \qquad \qquad  P^{\pi^{T-1,(h_{T-2},x_{T-1})}}_{\theta,T}(dx_{T})P^{\pi^{T-2,h_{T-2}}}_{\theta,T-1}(dx_{T-1}).
\end{align*}
We take now $t=T-2$. In this case,
\begin{multline*}
\rho_{T-2,T}(c_{T-2}(x_{T-2},\pi_{T-2}(h_{T-2}),\cdot),c_{T-1}(\diamond,\pi_{T-1}(h_{T-2},\diamond),\cdot),
c_{T}(\cdot,\pi_{T}(h_{T-2},\diamond,\cdot),\cdot),P^{\pi^{T-2,h_{T-2}}}_{T-1,T}) \\
=\widehat \rho_{T-2}\Big(
				\Big\{c_{T-2}(x_{T-2},\pi_{T-2}(h_{T-2}),\theta)+ \sigma_{T-2}\big(\rho_{T-1}\big(c_{T-1}(\diamond,\pi_{T-1}(h_{T-2},\diamond),\theta), c_{T}(\cdot,\pi_{T}(h_{T-2},\diamond,\cdot),\theta);\\
\delta_{{\theta}}\circledast P^{\pi^{T-1,(h_{T-2},\diamond)}}_{T,T|\Theta}(\cdot \| \diamond) \big);P^{\pi^{T-2,h_{T-2}}}_{T-1,T-1|\theta}\big),\theta \in \bTheta\Big\};P^{\pi^{T-2,h_{T-2}}}_{T-1,T|\Theta} \Big)\\
=\widehat \rho_{T-2}\Big(
				\Big\{c_{T-2}(x_{T-2},\pi_{T-2}(h_{T-2}),\theta)+ \sigma_{T-2}\big(\rho_{T-1}\big(c_{T-1}(\diamond,\pi_{T-1}(h_{T-2},\diamond),\theta), c_{T}(\cdot,\pi_{T}(h_{T-2},\diamond,\cdot),\theta); \\
\delta_{{\theta}}\circledast P^{\pi^{T-1,(h_{T-2},\diamond)}}_{T,T|\Theta}(\cdot \| \diamond) \big);P^{\pi^{T-2,h_{T-2}}}_{T-1,T-1|\theta}\big),\theta \in \bTheta\Big\};\xi^{\pi,h_{T-2}}_{T-2}\Big )\\
=\frac{1}{\varkappa}\ln \int_\bTheta \int_{\cX} \int_{\cX} e^{\varkappa(c_{T-2}(x_{T-2},\pi_{T-2}(h_{T-2}),\theta)+c_{T-1}(x_{T-1},\pi_{T-1}(h_{T-2},x_{T-1}),\theta) + c_{T}(x_T,\pi_{T}(h_{T-2},x_{T-1},x_{T}),\theta))}\\
P^{\pi^{T-1,h_{T-2},x_{T-1}}}_{\theta,T}(dx_{T})P^{\pi^{T-2,h_{T-2}}}_{\theta,T-1}(dx_{T-1})\xi^{\pi,h_{T-2}}_{T-2}(d\theta) \\
= \frac{1}{\varkappa}\ln E^\pi(e^{\varkappa(c_{T-2}(x_{T-2},\pi_{T-2}(h_{T-2}),\Theta)+c_{T-1}(X_{T-1},\pi_{T-1}(h_{T-2},X_{T-1}),\Theta) + c_{T}(X_T,\pi_{T}(h_{T-2},X_{T-1},X_T),\Theta))} \\
|\widehat H_{T-2}=h_{T-2}),
\end{multline*}
where the last equality follows from Lemma~\ref{cacy-cacy}.

Proceeding in the analogous way for $t=T-3,\ldots,1$ we finally obtain
\begin{align}\label{eq:RS-criterion}
\rho_{1,T}(c_{1}(x_{1},\pi_{1}(h_{1}),&\cdot),\ldots,c_{T}(\cdot,\pi_{T}(h_1,\cdot,\ldots,\cdot),\cdot),P^{\pi^{1,h_1}}_{2,T})  \nonumber \\
& = \frac{1}{\varkappa}\ln E^\pi\left (e^{\varkappa\sum_{k=1}^T\, c_{k}(X_{k},\pi_{k}(h_{k})), \Theta)}|\widehat H_{1}=h_{1}\right ),
\end{align}
which gives us the risk-sensitive criterion with entropic utility (cf. \cite{Baeuerle2014,DavisLleo2014}) .
\end{example}

\section{Recursive Risk Filters}\label{s:bayes-op}

Let us fix $t\in\set{1,\dots,T-1}$. To alleviate notation, for all $\pi \in \varPi$, we write for  fixed functions $Z_{t,s}(\cdot_{s-t},\cdot_1)\in \cZ_{t,s}$, $s=t,\ldots,T$.
Since $t$ is fixed, we will again simply write $Z_s$ instead of $Z_{t,s}$, for $s\in\cT_t$.
\begin{align}\label{policy-value-gen-tom}
v_t^{\pi}(h_t) &=\rho_t(Z_{t},Z_{t+1},\dots,Z_{T}; P^{\pi^{t,h_t}}_{t+1,T})\\
\widetilde v_{t+1}^{\pi,\theta}((h_t,x_{t+1}))&:=\rho_{t+1}\big(Z_{t+1}(x_{t+1},\cdot_1), \dots ,Z_{T}(x_{t+1},\cdot_{T-t-1},\cdot_1);\delta_{{\theta}}  \circledast P^{\pi^{t,h_t}}_{t+1,T|\Theta}(\,\cdot\, \| x_{t+1} )\big)\nonumber \\
&=\rho_{t+1}\big(Z_{t+1}(x_{t+1},\cdot_1), \dots ,Z_{T}(x_{t+1},\cdot_{T-t-1},\cdot_1);P^{\pi^{t,h_t}}_{\theta,t+1,T}(\{x_{t+1}\}\times \cdot)\delta_{ \theta}(\cdot)\big),          		
\end{align}
where for the last equality we used \eqref{eq:wr-delta}.

The quantity $v_t^{\pi}(h_t)$ evaluates the policy $\pi$ at the time $t$ and with the history $h_t$ in the original problem. 	

Recall that (cf. \eqref{three-eqs}) $P^{\pi^{t,h_t}}_{t+1,T,\Theta}=\xi^{\pi,h_t}_{t}$. Thus, the key equation \eqref{li-risk-trans-map-controlled} can be written more compactly as follows:
\begin{align}
v_t^{\pi}(h_t)&=\widehat \rho_t\Big(\Big\{Z_{t}(\theta)+ \sigma_t\big(
\rho_{t+1}\big(Z_{t+1}(\diamond,\cdot_1), \dots ,Z_{T}(\diamond,\cdot_{T-t-1},\cdot_1);\nonumber \\
&\qquad \qquad \qquad\qquad \delta_{{\theta}}  \circledast P^{\pi^{t,h_t}}_{t+1,T| \Theta}(\,\cdot\, \| \diamond )\big);P^{\pi^{t,h_t}}_{t+1,t+1|\theta}\big),\theta \in \bTheta\Big\};
P^{\pi^{t,h_t}}_{t+1,T,\Theta} \Big)\nonumber  \\
&=\widehat \rho_t\Big(
				\Big\{Z_{t}(\theta)+ \sigma_t\big(Z_{t+1}(\diamond,\cdot_1), \dots ,Z_{T}(\diamond,\cdot_{T-t-1},\cdot_1);\nonumber \\
&\qquad \qquad \qquad\qquad \delta_{{\theta}}  \circledast P^{\pi^{t,h_t}}_{t+1,T|\Theta}(\,\cdot\, \| \diamond )\big);P^{\pi^{t,h_t}}_{t+1,t+1|\theta}\big),\theta \in \bTheta\Big\};
			\xi^{\pi,h_t}_{t} \Big) \nonumber\\
&= \widehat \rho_t\Big(\Big\{Z_{t}(\theta)+ \sigma_t\big(\widetilde v_{t+1}^{\pi,\theta}((h_t,\diamond));P^{\pi^{t,h_t}}_{t+1,t+1|\theta}\big),\theta \in \bTheta\Big\};
				\xi^{\pi,h_t}_{t} \Big),\nonumber\\
&= \widehat \rho_t\Big(\Big\{Z_{t}(\theta)+ \sigma_t\big(\widetilde v_{t+1}^{\pi,\theta}((h_t,\diamond));P^{\pi^{t,h_t}}_{\theta,t+1}\big),\theta \in \bTheta\Big\};
				\xi^{\pi,h_t}_{t} \Big),\label{li-risk-trans-map-controlled-new-use}
\end{align}
with $\sigma_t$ given in \eqref{sigma-def-tom-original-0}, and where we used \eqref{ii-1-new}  in the last equality.

Note that in equation \eqref{li-risk-trans-map-controlled-new-use} we have $\widetilde v_{t+1}^{\pi,\theta}$ on the right hand side. Thus, this equation does not provide a convenient recursion for the quantities $v^\pi_t$. This leads us to the following concept,

\begin{definition}
A dynamic risk filter $\rho$ is called \textit{recursive} if it is satisfies the properties stated in Theorem \ref{t:filter-structure} and
\begin{equation*}
v_t^{\pi}(h_t)=\widehat \rho_t\Big(\Big\{Z_{t,t}(\theta)+ \sigma_t\big( v_{t+1}^{\pi}((h_t,\diamond));P^{\pi^{t,h_t}}_{\theta,t+1}\big),\theta \in \bTheta\Big\};
				\xi^{\pi,h_t}_{t} \Big),
\end{equation*}
for $t=T-1,\ldots,1$, with
\begin{equation*}
	v_T^{\pi}(h_T)=\widehat \rho_t\Big(\Big\{Z_{T,T}(\theta),\theta \in \bTheta\Big\};
		\xi^{\pi,h_T}_{T} \Big).
\end{equation*}
\end{definition}

\subsection{Examples of recursive dynamic risk-filters}
We will show here that risk-filters considered in Example~\ref{rem:notations} and Example~\ref{e:trials-tom} are recursive.
\medskip

In the case of the additive risk rewards, that is Example~\ref{rem:notations}, using \eqref{mult} we get
\begin{align}\label{eq:main-1}
v^\pi_t(h_t)&=\rho_t\Big(c_t(x_t,\pi_t(h_t),\cdot),c_{t+1}(\cdot,\pi_{t+1}(h_t,\cdot),\cdot),\nonumber \\
&\qquad \qquad \qquad \qquad \qquad \qquad  c_{t+2}(\cdot,\pi_{t+2}(h_t,\cdot,\cdot),\cdot),\cdots ,c_{T}(\cdot,\pi_{T}(h_t,\cdot,\ldots,\cdot),\cdot),P^{\pi^{t,h_t}}_{t+1,T}\Big )\nonumber\\
&=E^\pi\left[c_t(x_t,\pi_t(h_t),\Theta) + \sum_{s=t+1}^T c_s(\wh X_s, \pi_s(\wh X_{t+1},\ldots,\wh X_s), \Theta) \mid \wh H_t=h_t\right],
\end{align}
for $t\in\cT$, where again we use the standard convention that an empty sum is zero (i.e. $\sum_{s=t+T}^T\cdots =0$ in our case).

Given the above, we obtain that the risk filter considered in Example~\ref{rem:notations} is recursive:

\begin{lemma}\label{eq:hope}
Let $\widehat \rho_t$ and $\sigma_t$ be given as in \eqref{eq:main1} and \eqref{eq:main2}, respectively. We have
\begin{align}\label{eq:main-2-alt}
			&v^\pi_t(h_t)= \widehat \rho_t\Big(
			\Big\{c_t(x_t,\pi_t(h_t),\theta)+ \sigma_t\big( v_{t+1}^{\pi}((h_t,\cdot));P^{\pi^{t,h_t}}_{\theta,t+1}\big),\theta \in \bTheta\Big\};   \xi^{\pi,h_t}_{t}  \Big)
\end{align}
for $t=T-1,\ldots,1$, with
\begin{align}
			\label{eq:main-2-alt-T}
v^\pi_T(h_T)=\int _\bTheta\, c_T(x_T,\pi_T(h_T),\theta)\xi^{\pi,h_T}_{T}(d\theta).
\end{align}

\end{lemma}

\begin{proof}
Fix $t\in \{1,\ldots,T-1\}.$ First, using \eqref{eq:main-1} and the tower property of conditional expectations we get
\begin{align*}
v_t^\pi(\widehat H_t) &= E^\pi\left[c_t(\widehat X_t,\pi_t(\widehat H_t),\Theta) + \sum_{s=t+1}^T c_s(\widehat X_s, \pi_s^{t,h_t}(\widehat X_{t+1},\ldots,\widehat X_s), \Theta) \mid \widehat H_t\right]\\
 &= E^\pi\left[c_t(\widehat X_t,\pi_t(\widehat  H_t), \Theta) + E^\pi \left[\sum_{s=t+1}^T c_s(\widehat X_s, \pi_s(\widehat X_{t+1},\ldots,\widehat X_s),\Theta) \mid \widehat H_{t+1}\right]\mid \widehat H_t\right], \\
& = E^\pi[c_t(\widehat X_t,\pi_t(\widehat  H_t), \Theta) + v_{t+1}^\pi(\widehat H_{t+1}) \mid \widehat H_t],
\end{align*}
and so
\begin{align*}
v_t^\pi(h_t) = E^\pi[c_t(x_t,\pi_t(h_t), \Theta) + v_{t+1}^\pi(h_t,\widehat X_{t+1}) \mid \widehat H_t=h_t].
\end{align*}
Next, by Lemma~\ref{cacy-cacy}, we have
\begin{equation}\label{eq:add-recurrent-v}
v_t^\pi(h_t) = \int_{\boldsymbol{\Theta}} \int_\cX \left(c_t(x_t,\pi_t(h_t), \theta) + v_{t+1}^\pi(h_{t}, x_{t+1}) \right)
P_{\theta,t+1}^{\pi^{t, h_t}}(dx_{t+1})\; \xi_t^{\pi^{t,h_t}}(d\theta),
\end{equation}
and by taking into account the form of $\widehat \rho_t$ and $\sigma_t$ as in \eqref{eq:main1} and \eqref{eq:main2}, respectively, we obtain \eqref{eq:main-2-alt}. Finally, \eqref{eq:main-2-alt-T} is a direct consequence of \eqref{eq:main-1} and Lemma~\ref{cacy-cacy}.

\end{proof}

In the case of the  risk-sensitive rewards, that is Example  \ref{e:trials-tom}, the recursiveness of $\rho$ can be demonstrated in a way analogous to the above.

\section{Risk-Averse Control Problem}\label{s:RAC-problem}

	Let $v_1^\pi$ be as in \eqref{policy-value-gen-tom}. The control problem is to find
	\begin{equation}\label{min}
	\min_{\pi\in \Pi} v_1^{\pi}(h_1),
	\end{equation}
	as well as the optimal policy, say $\pi^*$, for which
$ v_1^{\pi^*}(h_1)=\min_{\pi\in \Pi} v_1^{\pi}(h_1).$
Note that given our set-up, an optimal policy does exist because the set $\Pi$ is finite. However, we are interested in seeking an optimal policy in the class of quasi-Markov policies.

\begin{definition} \label{def:quasi-markov-policy}
		A policy $\pi \in \varPi$ is  \emph{quasi-Markov (QMP)} if \[\pi_t(h_t)=\phi_t(x_t,\xi^{\pi,h_t}_t)\]
			for some function $\phi_t\, :\, \cX \times \bTheta \rightarrow \mathcal{U}$, $t=1,\ldots,T$.
	\end{definition}

\subsection{The Bayes Operator}\label{s:Bayes-Operator2}

At each time $t$ and for every policy $\pi$ and history $h_t$, the measure
	$P^{\pi^{t,h_t}}_{t+1}$ (cf. \eqref{eq:new-measure-1}) describes the conditional joint distribution of the pair $(\widehat X_{t+1},\Theta)$ in $\cX \times \bTheta$.
	
This measure admits two natural disintegrations. One of them is already obtained from \eqref{eq:I-T-integrated2}, repeated here :
\begin{align*}
P^{\pi^{t,h_t}}_{t+1}(B\times D)&=P^{\pi^{t,h_t}}_{t+1,T}(B\times \mathcal{X}^{T-t-1}\times D)=\int_D\, {P^{\pi^{t,h_t}}_{\theta,t+1,T}(B\times \mathcal{X}^{T-t-1})}\;\xi^{\pi,h_t}_{t}(d\theta)\\
&= P^\pi[\widehat X_{t+1}\in B,\Theta \in D \, | \, \widehat H_t = h_t],
\end{align*}
where $\xi^{\pi,h_t}_{t}\in \cP(\boldsymbol\Theta)$, is given as (cf. \eqref{eq:xi-tCond} and \eqref{i-1})
$\xi^{\pi,h_t}_{t}(D)=P^\pi[\Theta \in D \, | \, \widehat H_t = h_t]=P^{\pi^{t,h_t}}_{t+1,\Theta}(D).$
One can also disintegrate $P^{\pi^{t,h_t}}_{t+1}$ into its marginal on $\cX$, say\footnote{For simplicity of notations, we write $P^{\pi^{t,h_t}}_{t+1,X}$ instead of more coherent notation $P^{\pi^{t,h_t}}_{t+1,X_{t+1}}$. Similar remark applies to the kernel $P^{\pi^{t,h_t}}_{t+1|X}$.}  $ P^{\pi^{t,h_t}}_{t+1,X}$, and the corresponding stochastic kernel, say $P^{\pi^{t,h_t}}_{t+1|X}$ from $\cX$ to $\boldsymbol \Theta$. That is, for any $B\times D\subset \cX\times\boldsymbol{\Theta}$,
\begin{align}
P^{\pi^{t,h_t}}_{t+1}(B\times D) & = (P^{\pi^{t,h_t}}_{t+1, X} \circledast P^{\pi^{t,h_t}}_{t+1|{X}}) (B\times D), \nonumber\\
& =\int_B\, P^{\pi^{t,h_t}}_{t+1|x}(D)P^{\pi^{t,h_t}}_{t+1,X}(dx) \nonumber \\
& =P^\pi[\widehat X_{t+1}\in B,\Theta \in D \, | \, \widehat H_t = h_t],
\label{eq:I-T-integrated-Tom}
\end{align}
where we used the simplified notation $P^{\pi^{t,h_t}}_{t+1|x}(D)$ for $P^{\pi^{t,h_t}}_{t+1|{X}}(x,D).$

The kernel $P^{\pi^{t,h_t}}_{t+1|x}$ is the Bayes operator which describes the dynamics of the belief states, as documented in the next result.

\begin{lemma}\label{Bayes-dyn}
For $t=1,\ldots,T-1$, $h_t\in H_t$, $x_{t+1}\in \cX$ and $D\subset \bTheta$, we have
\begin{align}\label{Bayes-trb}
\xi^{\pi,(h_t,x_{t+1})}_{t+1}(D)&=P^{\pi^{t,h_t}}_{t+1\mid x_{t+1}}(D) \nonumber \\
& = \int_{D} \frac{K_\theta(x_{t+1}|x_{t},\pi_{t}(h_t))}{P^{\pi^{t,h_t}}_{t+1}[\{x_{t+1}\}\times \bTheta]} \, \xi^{\pi,h_t}_{t}(d\theta),
\end{align}
where
\begin{equation}\label{eq:rec-0}
\xi^{\pi,x_{1}}_{1}(\theta)=\xi_{1}(\theta).
\end{equation}
\end{lemma}
\begin{proof}
First, note that
\begin{equation}\label{Q}
P^{\pi^{t,h_t}}_{t+1,X}(B)=P^\pi[\widehat X_{t+1}\in B \, | \, \widehat H_{t} = h_t].
\end{equation}

Take $B=\{x_{t+1}\}$. Then, using \eqref{eq:I-T-integrated-Tom} and \eqref{Q}, we obtain
		\[P^\pi[\widehat X_{t+1}=x_{t+1},\Theta \in D \, | \, \widehat H_t = h_t]=
 P^{\pi^{t,h_t}}_{t+1|x_{t+1}}(D)P^\pi[\widehat X_{t+1}=x_{t+1} \, | \, \widehat H_{t} = h_t],\]
		and thus
\begin{align*}
P^{\pi^{t,h_t}}_{t+1|x_{t+1}}(D)&=\frac{P^\pi[\widehat X_{t+1}=x_{t+1}, \Theta \in D \, | \, \widehat H_t = h_t]}{P^\pi[\widehat X_{t+1}=x_{t+1} \, | \, \widehat H_{t} = h_t]}\nonumber \\  & =P^\pi[\Theta \in D \, | \, \widehat H_{t+1} = (h_t,x_{t+1})]=\xi^{\pi,(h_t,x_{t+1})}_{t+1}(D),
\end{align*}
which proves the first equality in \eqref{Bayes-trb}. The second one follows from the following chain of equalities,
\begin{align*}
\xi^{\pi,(h_t,x_{t+1})}_{t+1}(\theta) & =P^\pi[\Theta =\theta \, | \, \widehat H_{t+1} = (h_t,x_{t+1})] \\
		& =P^\pi[\Theta =\theta \, | \, \widehat H_{t} = h_t]\frac{P^\pi[\widehat X_{t+1}=x_{t+1}\, |\,\boldsymbol \Theta =\theta , \widehat H_{t} = h_t] }{P^\pi[\widehat X_{t+1}=x_{t+1}\, |\, \widehat H_{t} = h_t]}\\
		& =\xi^{\pi,h_t}_{t}(\theta)\frac{K_\theta(x_{t+1}|x_{t},\pi_{t}(h_t))}{P^{\pi^{t,h_t}}_{t+1}[\{x_{t+1}\}\times \bTheta]},
\end{align*}
where in last equality we used \eqref{eq:P-theta-t+1} and that $P^\pi_\theta (B) = P^\pi(B|\Theta=\theta)$.
\end{proof}

\subsection{Optimal control problem corresponding to Example~\ref{ex:additive-reward}}\label{sec:add}
In this section we will study the optimal control problem corresponding to the Example~\ref{ex:additive-reward} classical additive reward case, that will serve as the base for the general case.
In what follows, we denote by $(x,\xi)$ an element of the set $\mathcal{X}\times \mathcal{P}(\boldsymbol \Theta).$

Recall \eqref{eq:main-1}. Accordingly, we have for $t=T$ 
\begin{align}
\label{eq:cond-exp-1-sol-T}
v^\pi_T(h_T)&=\rho_{T,T}\Big(c_T(x_T,\pi_T(h_T),\cdot),P^{\pi^{T,h_T}}_{T+1,T}\Big) \nonumber \\
&=\int_{\boldsymbol \Theta}c_T(x_T,\pi_T(h_T),\theta))\; P^{\pi^{T,h_T}}_{T+1,T}(d\theta) \nonumber \\
&=\int_{\boldsymbol \Theta}c_T(x_T,\pi_T(h_T),\theta)\; \xi^{\pi,h_T}_T(d\theta)\\
& = E^\pi\Big (c_T(x_T,\pi_T(h_T), \Theta) \mid \widehat H_T=h_T\Big ). \nonumber
\end{align}

Thus, observing  that $\xi^{\pi,h_T}_T$, does not depend on $\pi_T$, letting $x_T=x$ and $\xi^{\pi,h_T}_T=\xi$, we compute the candidate-optimal quasi-Markov control $\phi_T$ as
\begin{equation}\label{opt-T}
\phi_T(x,\xi)=\argmin_{u\in \cU}\, \int_{\boldsymbol \Theta}\, c_T(x,u,\theta)\;
 \xi(d\theta).
\end{equation}

We define the Bellman function at time $t=T$:
\begin{equation}\label{Bell-opt-T}
V_T(x,\xi)= \min_{u\in \cU}\, \int_{\boldsymbol \Theta}\, c_T(x,u,\theta)\,
 \xi(d\theta)=\int_{\boldsymbol \Theta}\, c_T(x,\phi_T(x,\xi),\theta)\,
 \xi(d\theta).
\end{equation}

Now, we proceed to time $t=T-1$. Noting that $\xi^{\pi,h_{T-1}}_{T-1}$, does not depend on $\pi_{T-1}$, letting $x_{T-1}=x$ and $\xi^{\pi,h_{T-1}}_{T-1}=\xi$, we compute the candidate-optimal quasi-Markov control $\phi_{T-1}$ as
\begin{equation}
\label{opt-T-1}
\phi_{T-1}(x,\xi)=\argmin_{u\in \cU}\, \int_{\boldsymbol \Theta}\, \Big(c_{T-1}(x,u,\theta)+\int_{\mathcal{X}}V_{T}(x_{T},{\wt \xi^{u,x_{T},\xi}_{T}})K_{\theta}(dx_{T}|x,u)
 \Big )\, \xi(d\theta),
\end{equation}
where (cf. \eqref{Bayes-trb})
\begin{equation}\label{xi-T}
\wt \xi^{u,x_{T},\xi}_{T}(\theta)=\xi(\theta)\frac{K_\theta(x_{T} |  x,u)}{\int_{\bTheta}\, K_\theta(x_{T} |  x,u)\, \xi(d\theta)}.
\end{equation}
The corresponding Bellman function is
\begin{align*}
V_{T-1}(x,\xi)&= \min_{u\in \cU}\, \int_{\boldsymbol \Theta}\, \Big(c_{T-1}(x,u,\theta)+\int_{\mathcal{X}}V_{T}(x_{T},{\wt \xi^{u,x_{T},{\xi}}_{T}})K_{\theta}(dx_{T}|x,u)
 \Big )\, \xi(d\theta)\nonumber \\ &= \int_{\boldsymbol \Theta}\, \Big(c_{T-1}(x,\phi_{T-1}(x,\xi),\theta)+\int_{\mathcal{X}}V_{T}(x_{T},{\wt \xi^{\phi_{T-1}(x,\xi),x_{T},\xi}_{T}})K_{\theta}(dx_{T}|x,\phi_{T-1}(x,\xi))
 \Big )\, \xi(d\theta).
\end{align*}
Following this pattern, we arrive at the dynamic programming (DP) backward recursion:
\begin{equation}\label{opt-T-1a}
V_t(x,\xi)=\min_{u\in \cU}\, \int_{\boldsymbol \Theta}\, \Big(c_{t}(x,u,\theta)+\int_{\mathcal{X}}V_{t+1}(x_{t+1},{\wt \xi^{u,x_{t+1},\xi}_{t+1}})K_{\theta}(dx_{t+1}|x,u)
 \Big )\, \xi(d\theta),\quad t\in\cT,
\end{equation}
where (cf. \eqref{Bayes-trb})
\begin{equation}\label{xi-t+1}
\wt \xi^{u,x_{t+1},\xi}_{t+1}(\theta)=\xi(\theta)\frac{K_\theta(x_{t+1}|x,u)}{\int_{\bTheta}\, K_\theta(x_{t+1}|x,u)\, \xi(d\theta)},
\end{equation}
and
\begin{equation}\label{term-v}
V_{T+1}\equiv 0.
\end{equation}
Note that \eqref{opt-T-1a} is a counterpart of \eqref{eq:main-2-alt}.

Accordingly, for $t=1,\ldots,T$ we define the candidate-optimal quasi-Markov control $\phi_{t}$ as
\begin{equation}\label{opt-T-1-strategy}
\phi_{t}(x,\xi)=\argmin_{u\in \cU}\, \int_{\boldsymbol \Theta}\, \Big(c_{t}(x,u,\theta)+\int_{\mathcal{X}}V_{t+1}(x_{t+1},{\wt \xi^{u,x_{t+1},\xi}_{t+1}})K_{\theta}(dx_{t+1}|x,u)
 \Big )\, \xi(d\theta).
\end{equation}

Recall that $\xi_1$ is a given prior distribution for $\Theta$. Also, recall that $h_1=x_1$.

Next, define a policy $\pi^*$ as follows,
\begin{align}\label{pistar}
\pi^*_1(h_1)&=\phi_1(x_1, \xi_1)\nonumber \\
\pi^*_t(h_t)&=\phi_t(x_t,\wh \xi^{\pi^*,h_t}_t),\quad t=2,\ldots,T,
\end{align}
where
\begin{align}\label{xi-hat}
\wh \xi^{\pi^*,h_2}_2=\wt \xi^{\pi^*_{1}(h_{1}),x_2,\xi_1}_{2}, \quad
\wh \xi^{\pi^*,h_3}_3=\wt \xi^{\pi^*_{2}(h_{2}),x_3,\wh \xi^{\pi^*,h_2}_2}_{3}, \ldots
\end{align}

The next result is the optimality verification theorem.

\begin{theorem}\label{thm:ver}
We have,
$$
\min _{\pi \in \Pi} v^\pi_1(h_1) = v^{\pi^*}_1(h_1)=V_1(x_1,\xi_1).
$$
\end{theorem}

\begin{proof}
Let $\pi\in \Pi$.
For $t=T$ we have
\[
v^\pi_T(h_T)\geq V_T(x_T,\xi^{\pi,h_T}_T)=\int_{\boldsymbol \Theta}\, c_T(x_T,\phi_T(x_T,\xi^{\pi,h_T}_T),\theta)\, \xi^{\pi,h_T}_T(d\theta).
 \]
For $t=T-1$, using the above, the recursion \eqref{eq:add-recurrent-v}, and \eqref{Bayes-trb}, we have
 \begin{align*}
v^\pi_{T-1}&(h_{T-1})
=\int_{\boldsymbol \Theta}\,\Bigg ( c_{T-1}(x_{T-1},\pi_{T-1}(h_{T-1}),\theta)
+ \int_{\mathcal{X}}\,  v_T^\pi(h_T)K_{\theta}\big (dx_{T}|x_{T-1},\pi_{T-1}(h_{T-1})\big )\Bigg )\xi^{\pi,h_{T-1}}_{T-1}(d\theta)\\
&\geq \int_{\boldsymbol \Theta}\,\Bigg ( c_{T-1}(x_{T-1},\pi_{T-1}(h_{T-1}),\theta)
 + \int_{\mathcal{X}}\,  V_T(x_T,\xi^{\pi,h_{T}}_{T})K_{\theta}\big (dx_{T}|x_{T-1},\pi_{T-1}(h_{T-1})\big )\Bigg )\xi^{\pi,h_{T-1}}_{T-1}(d\theta)\\
&\geq \int_{\boldsymbol \Theta}\,\Bigg ( c_{T-1}(x_{T-1},\phi_{T-1}(x_{T-1},\xi^{\pi,h_{T-1}}_{T-1}),\theta) \\
 & \qquad  + \int_{\mathcal{X}}\,{ V_T(x_T,\wt \xi^{\phi_{T-1}(x_{T-1},\xi^{\pi,h_{T-1}}_{T-1}),x_{T},\xi^{\pi,h_{T-1}}_{T-1}}_{T})}K_{\theta}\big (dx_{T}|x_{T-1},\phi_{T-1}(x_{T-1},\xi^{\pi,h_{T-1}}_{T-1})\big )\Bigg )\xi^{\pi,h_{T-1}}_{T-1}(d\theta) \\
&  =V_{T-1}(x_{T-1}, \xi^{\pi,h_{T-1}}_{T-1}).
\end{align*}
Likewise, for $t=1,\ldots,T-2$, we have
\begin{multline*}
v^\pi_t(h_t)\geq V_t(x_t,\xi^{\pi,h_t}_t)=\int_{\boldsymbol \Theta}\,\Bigg ( c_{t}(x_{t},\phi_{t}(x_{t},\xi^{\pi,h_{t}}_{t}),\theta) \\
+ \int_{\mathcal{X}}\, V_{t+1}(x_{t+1},\phi_{t+1}(x_{t+1},\wt \xi^{\phi_{t}(x_{t},\xi^{\pi,h_{t}}_{t}),x_{{t+1}},\xi^{\pi,h_{t}}_{t}}_{{t+1}}))K_{\theta}\big (dx_{{t+1}}|x_{t},\phi_{t}(x_{t},\xi^{\pi,h_{t}}_{t})\big )\Bigg )\xi^{\pi,h_{t}}_{t}(d\theta).
\end{multline*}
Now, if $\pi $ and $\xi^{\pi,h_{t}}_{t}$, $t\in\cT$, above are replaced with  $\pi^*$ and $\wh{\xi}^{\,\pi^*\!,h_{t}}_{t}$, $t\in\cT$, respectively, then the inequalities above become equalities, proving that $\pi^* $ is an optimal strategy.

\end{proof}

Recalling \eqref{eq:main1} and \eqref{eq:main2}, we note that the key DP recursion \eqref{opt-T-1} can be written as
\[
V_t(x,\xi)=\min_{u\in \mathcal{U}}\, \widehat \rho_t \Big(\Big\{ c_{t}(x,u,\theta)
+{\sigma_t(V_{t+1}(\cdot,{\wt \xi^{u,\cdot,\xi}_{t+1}});K_\theta(x,u))} ,\theta \in \bTheta\Big \};   \xi  \Big),
\]
subject to \eqref{xi-t+1} and \eqref{term-v}.

\begin{example}\label{ex:risk-averse-unknown}
	
We remark that the optimal control problem considered in this section can be cast in the classical optimal investment and consumption problem, now also subject to model uncertainty. Namely, consider an investor with initial capital $\bar x_1$, who can invest in $d$ assets, with $\bar X_t$ denoting the portfolio value at time $t$, which of course is observed by the investor.  The investor rebalances the portfolio at each time $t$, following a self-financing trading strategy (policy) $\pi$, that may satisfy additional trading constrains, such as short selling constrains, turn over constraints, etc. The investor is also allowed to consume at each time $t$ part of the wealth, say $z_t$, that does not exceed $\bar X_t$. We postulate that the investor maximizes the expected utility of consumptions and terminal wealth using the utility functions $V^\beta$ and $U^\gamma$, respectively, where $\beta, \gamma\in\bR$ stand for risk aversion-parameters. We refer the reader to \cite[Section 4]{BauRie2017} for detailed formulation of this problem  in the MDP framework.

Additionally, we assume that the investor faces the Knightian uncertainty about the model of the underlying assets, described in terms of a (finite) parametric set $\boldsymbol \Lambda\subset \bR^m$; see \cite{BCCCJ2019} for an overview of MDPs under Knightian uncertainty.

Moreover, we suppose that the investor is also uncertain about her risk aversion parameters $(\beta, \gamma)\in \boldsymbol \Gamma\subset \bR^2$. We emphasize that this additional feature of an unknown risk aversion parameter is practically important. Generally speaking, it is difficult to determine the investor's risk aversion parameter, which is well documented in the behavioral finance literature.  This becomes especially relevant in the context of fast-growing robo-advising industry that typically deals with unsophisticated investors, and which establishes investor's risk preferences without human intervention. At each time $t$, the investor reports through process $Y_t$ her subjective degree of happiness about the performance of her investment. For example, one can take $Y_t$ to be a Bernoulli random variable with $Y_t=1$ corresponding to happy and $Y_t=0$ meaning unhappy about her investment, and then follow a similar setup to the clinical trials Example~\ref{ex:3.6} and incorporate the uncertainty about $(\beta,\gamma)$ into the original MDP formulation.
	
Now we consider the observed state process $X_t=(\bar X_t,Y_t)$, and we take $\theta=(\beta,\gamma,\lambda)\in \boldsymbol{\Theta}=\boldsymbol{\Gamma}\times\boldsymbol{\Lambda}$ representing the model uncertainty in this model. Consequently, we define the cost functionals
\begin{align*}
c_t & = V^\beta(z_t(\bar x_t,\pi_t)) + F(y_t, \pi_t,\beta, \gamma), \ \  t=1,\ldots,T-1, \\
c_T & =  V^\beta(z_T(\bar x_T,\pi_T)) + U^\gamma(\bar x_T),
\end{align*}
where $F$ is a penalty for `deviating' from the true risk-aversion parameters.
Using the expectation as the risk functional (cf. Example~\ref{ex:additive-reward}), the problem \eqref{min} becomes the optimal investment and consumption problem. Theorem~\ref{thm:ver} gives the solution to this problem.
Detailed model specification and analysis is beyond the scope of this work and it will be addressed in future works.
\end{example}

\subsection{Optimal control problem corresponding to Example~\ref{e:trials-tom}}\label{sec:med}
We will present the solution to the optimal control problem for the clinical trials example with the risk-sensitive criterion.
Namely, for $t\in\cT$, we consider
\begin{align}\label{eq:RS-criterion-a}
v^\pi_{t}(h_t) & =\rho_t\Big(c_t(x_t,\pi_t(h_t),\cdot),c_{t+1}(\cdot,\pi_{t+1}(h_t,\cdot),\cdot),\nonumber \\
&\qquad \qquad c_{t+2}(\cdot,\pi_{t+2}(h_t,\cdot,\cdot),\cdot),\cdots ,c_{T}(\cdot,\pi_{T}(h_t,\cdot,\ldots,\cdot),\cdot),P^{\pi^{t,h_t}}_{t+1,T}\Big) \nonumber \\
& = \frac{1}{\varkappa}\ln E^\pi\left (\exp\left(\varkappa\Big (c_t(x_t,\pi_t(h_t),\Theta)+\sum_{k=t+1}^Tc_{k}(\widehat X_{k},\pi_{k}(h_t,\widehat X_{t+1},\ldots,\widehat X_k),\Theta)\Big )\right)\bigg |\widehat H_{t}=h_{t}\right ) \\
& = \frac{1}{\varkappa}\ln w^\pi_t(h_t).
\end{align}
It is clear that  problem \eqref{min} is equivalent to the following problem
\begin{equation}\label{min1}
	\min_{\pi\in \Pi} w_1^{\pi}(h_1),
	\end{equation}
For $t=T$ we have
\begin{align}\label{eq:cond-exp-1-sol-T-w}
w^\pi_T(h_T)= E^\pi\Big (\exp\left(\varkappa\, c_T(x_T,\pi_T(h_T), \Theta)\right)\big| \widehat H_T=h_T\Big )
=\int_{\boldsymbol \Theta}e^{\varkappa\, c_T(x_T,\pi_T(h_T),\theta)} \xi_T^{\pi, h_T}(d\theta).
\end{align}
As above, we denote by $(x,\xi)$ an element of the set $\mathcal{X}\times \mathcal{P}(\boldsymbol \Theta).$ Thus, observing  that $\xi^{\pi,h_T}_T$  does not depend on $\pi_T$, and letting  $x_T=x$ and $\xi_T=\xi$, we compute the candidate optimal quasi-Markov control $\varphi_T$ as
\begin{equation}\label{opt-T-w}
\varphi_T(x,\xi)=\argmin_{u\in \cU}\, \int_{\boldsymbol \Theta}\, e^{\varkappa\, c_T(x,u,\theta)}\,
 \xi(d\theta).
\end{equation}
We define the Bellman function at time $t=T$ as
\begin{equation}\label{Bell-opt-T-w}
W_T(x,\xi)= \min_{u\in \mathcal{U}}\, \int_{\boldsymbol \Theta}\, e^{\varkappa\, c_T(x,u,\theta)}\,
 \xi(d\theta)=\int_{\boldsymbol \Theta}\, e^{\varkappa\, c_T(x,\varphi_T(x,\xi),\theta)}\,
 \xi(d\theta).
\end{equation}
Now, we proceed to time $t=T-1$. Given $x_{T-1}=x$ and $\xi_{T-1}=\xi$ we compute the candidate optimal quasi-Markov control $\varphi_{T-1}$ as
\begin{equation}\label{opt-T-1-w}
\varphi_{T-1}(x,\xi)=\argmin_{u\in \cU}\, \int_{\boldsymbol \Theta}\int_{\mathcal{X}}e^{\varkappa\, c_{T-1}(x,u,\theta)}W_{T}(x_{T},{\wt \xi^{\,u,x_{T},\xi}_{T}})K_{\theta}(dx_{T}|x,u)
 \, \xi(d\theta),
\end{equation}
where $\wt \xi^{\,u,x_{T},\xi}_{T}$ is given by \eqref{xi-T}.
The corresponding Bellman function is
\begin{align}\label{Bell-opt-T-1-w}
W_{T-1}(x,\xi)&= \min_u\, \int_{\boldsymbol \Theta}\int_{\mathcal{X}}e^{\varkappa\, c_{T-1}(x,u,\theta)}W_{T}(x_{T},\wt \xi^{\,u,x_{T},\xi}_{T})\;K_{\theta}(dx_{T}|x,u)
 \; \xi(d\theta)\nonumber \\ &= \int_{\boldsymbol \Theta}\int_{\mathcal{X}}e^{\varkappa\, c_{T-1}(x,\varphi_{T-1}(x,\xi),\theta)}W_{T}(x_{T},{\wt \xi^{\,\varphi_{T-1}(x,\xi),x_{T},\xi}_{T}})\;K_{\theta}(dx_{T}|x,\varphi_{T-1}(x,\xi))
 \; \xi(d\theta).
\end{align}
Following this pattern, we arrive at the DP backward recursion:
\begin{equation}\label{opt-T-1-wa}
W_t(x,\xi)=\min_u\, \int_{\boldsymbol \Theta}\int_{\mathcal{X}}e^{\varkappa\, c_{t}(x,u,\theta)}W_{t+1}(x_{t+1},{\wt \xi^{\,u,x_{t+1},\xi}_{t+1}})\;K_{\theta}(dx_{t+1}|x,u)
 \; \xi(d\theta),\quad t\in\cT,
\end{equation}
where as in the previous example $\wt \xi^{u,x_{t+1},\xi}_{t+1}(\{\theta\})$ is given by \eqref{xi-t+1}.
and
$W_{T+1}\equiv 1.$
Note that \eqref{opt-T-1-wa} is a counterpart of \eqref{eq:main-2-alt}.

Accordingly, for $t\in\cT$, we define the candidate-optimal quasi-Markov control $\varphi_{t}$ as
\begin{equation}\label{opt-T-1-strategy-a}
\varphi_{t}(x,\xi)=\arg\min_u\, \int_{\boldsymbol \Theta}\int_{\mathcal{X}}e^{\varkappa\, c_{t}(x,u,\theta)}W_{t+1}(x_{t+1},{\wt \xi^{\,u,x_{t+1},\xi}_{t+1}})\;K_{\theta}(dx_{t+1}|x,u)
 \; \xi(d\theta),
\end{equation}
with $\xi_1$ being the given prior distribution for $\Theta$, and $h_1=x_1$.

The policy $\pi^*$ is defined by analogy to \eqref{pistar}. The following verification theorem can proved in a way analogous to the proof of Theorem \ref{thm:ver}, so we skip its proof.

\begin{theorem}
The following hold true
\begin{equation*}
\min _{\pi \in \Pi} v^\pi_1(h_1) = v^{\pi^*}_{1}(h_1)= \frac{1}{\varkappa} \ln\big(W_1(x_1,\xi_1)\big).
\end{equation*}
\end{theorem}

We emphasise that the key DP recursion \eqref{opt-T-1-w} may be written as
\[
W_t(x,\xi)=\min_{u\in \mathcal{U}}\, \widehat \rho_t \Big(\big\{ c_{t}(x,u,\theta)+\sigma_t\big(W_{t+1}(\cdot,{\wt \xi^{\,u,\cdot,\xi}_{t+1}});
K_\theta(x,u)\big) ,\theta \in \bTheta\big \}; \xi \Big),
\]
where for a function $f$ on $\bTheta$, $\xi\in\cP(\bTheta)$, and a function $h$ on $\cX$, we have
\[
\widehat \rho _t\big (\big \{ f(\theta),\theta \in \bTheta\big \};\xi\big ) =\int_\bTheta e^{\varkappa f(\theta)}\xi(d\theta),
\]
and where
\[
\sigma_t\big(h;K_\theta(x,u)\big) =\frac{1}{\varkappa}\ln \int_{\mathcal{X}} h(x_{t+1})K_{\theta}(dx_{t+1}|x,u).
\]

\subsection{Solution of the optimal control problem for general recursive risk filters}
Let $\rho$ be a recursive risk filter, and let
\begin{align}\label{eq:main-gen}
&v^\pi_t(h_t)=\rho_t\Big(c_t(x_t,\pi_t(h_t),\cdot),c_{t+1}(\cdot,\pi_{t+1}(h_t,\cdot),\cdot),c_{t+2}(\cdot,\pi_{t+2}(h_t,\cdot,\cdot),\cdot), \cdots\nonumber \\
&\qquad \qquad \qquad \qquad \qquad  \cdots ,c_{T}(\cdot,\pi_{T}(h_t,\cdot,\ldots,\cdot),\cdot),P^{\pi^{t,h_t}}_{t+1,T}\Big ),
\quad t\in\cT.
\end{align}
Consider the general problem \eqref{min} with $v^\pi_t(h_t)$ as in \eqref{eq:main-gen}.

Using reasoning analogous to the one employed in Sections \ref{sec:add} and \ref{sec:med} one can prove the following result, proof of which we omit here.

\begin{theorem}\label{thm:RA}
There exist operators $\widehat \rho_t$ and $\sigma_t$, $t\in\cT$, and a function $V^*$ such that for the functions $v^*_t$ defined recursively as
\begin{align*}
v^*_{T+1}(x) & = V^*(x),\ x\in \mathcal{X}, \\
v^*_t(x,\xi) & = \min_{u\in \mathcal{U}}\, \widehat \rho_t \Big(\Big\{ c_{t}(x,u,\theta)+\sigma_t(v^*_{t+1}(\cdot,{\wt \xi^{\,u,\cdot,\xi}_{t+1}});
K_\theta(x,u)) ,\theta \in \bTheta\Big \};   \xi  \Big),\\
& \qquad \quad \qquad \quad t=T-1,\ldots,1,\quad x\in \mathcal{X},\ \xi\in \mathcal{P}(\boldsymbol \Theta),
\end{align*}
subject to
\[
\wt \xi^{\,u,x',\xi}_{t+1}(\theta)=\xi(\theta)\frac{K_\theta(x'|x,u)}{\int_{\bTheta}\, K_\theta(x'|x,u)\, \xi(d\theta)},\quad t\in\cT, \  x,x'\in \mathcal{X},\ \xi\in \mathcal{P}(\boldsymbol \Theta),
\]
we have that
\begin{equation*}
\min _{\pi \in \Pi} v^\pi_1(h_1)= v^*_1(x_1,\xi_1).
\end{equation*}
 Moreover, the policy $\pi^*$ defined as in \eqref{pistar} and \eqref{xi-hat}, with the $\phi_t$'s given as
\begin{multline*}
\phi_t(x,\xi)=\argmin_{u\in \mathcal{U}}\, \widehat \rho_t \Big(\Big\{ c_{t}(x,u,\theta)+\sigma_t(v^*_{t+1}(\cdot,{\wt \xi^{u,\cdot,\xi}_{t+1}});K_\theta(x,u)) ,\theta \in \bTheta\Big \};   \xi  \Big),\\
 t=1,\ldots,T-1,\ x\in \mathcal{X},\ \xi\in \mathcal{P}(\boldsymbol \Theta),
\end{multline*}
is an optimal policy, that is

\begin{equation}
\min _{\pi \in \Pi} v^\pi_1(h_1) = v^{\pi^*}_{1}(h_1).
\end{equation}
The form of the operators $\widehat \rho_t$ and $\sigma_t$, $t\in\cT$, depends on the form of $\rho$, and it can be explicitly written in terms of $\rho$.
\end{theorem}

\bigskip \noindent
\section*{Acknowledgment}
The research of Andrzej Ruszczy\'{n}ski has benefited from partial support from National Science Foundation Award DMS-1907522 and by the Office of Naval Research Award N00014-21-1-2161.
Tomasz R. Bielecki and Igor Cialenco acknowledge support from the National Science Foundation grant DMS-1907568.

\bibliographystyle{alpha}
\newcommand{\etalchar}[1]{$^{#1}$}

\end{document}